\documentclass[10pt, reqno]{amsart}
\usepackage{graphicx, amssymb, amsmath, amsthm}
\usepackage{hyperref}
\numberwithin{equation}{section}

\let\Re=\undefined\DeclareMathOperator*{\Re}{Re}
\let\Im=\undefined\DeclareMathOperator*{\Im}{Im}

\newcommand{\ld}{\mathcal{L}_a}

\DeclareMathOperator*{\supp}{supp}

\newcommand{\R}{\mathbb{R}}
\newcommand{\N}{\mathbb{N}}
\newcommand{\C}{\mathbb{C}}

\newcommand{\eps}{\varepsilon}

\newcommand{\Z}{\mathbb{Z}}

\newcommand{\qtq}[1]{\quad\text{#1}\quad}

\newtheorem{theorem}{Theorem}[section]

\newtheorem{lemma}[theorem]{Lemma}

\newtheorem{corollary}[theorem]{Corollary}

\newtheorem{proposition}[theorem]{Proposition}

\theoremstyle{definition}
\newtheorem{definition}[theorem]{Definition}

\theoremstyle{remark}

\makeatletter
\newcommand{\Extend}[5]{\ext@arrow0099{\arrowfill@#1#2#3}{#4}{#5}}
\makeatother

\begin{document}
\title[Energy-critical NLS with inverse-square potential]{The energy-critical NLS with inverse-square potential}

\author{R. Killip}
\address{Department of Mathematics, UCLA}
\email{killip@math.ucla.edu}

\author{C. Miao}
\address{Institute of Applied Physics and Computational Mathematics, Beijing 100088}
\email{miao\_changxing@iapcm.ac.cn}

\author{M. Visan}
\address{Department of Mathematics, UCLA}
\email{visan@math.ucla.edu}

\author{J. Zhang}
\address{Department of Mathematics, Beijing Institute of Technology, Beijing 100081, China; Beijing Computational Science Research Center, Beijing}
\email{zhang\_junyong@bit.edu.cn}

\author{J. Zheng}
\address{Universit\'e Nice Sophia-Antipolis, 06108 Nice Cedex 02, France}
\email{zhengjiqiang@gmail.com, zheng@unice.fr}

\begin{abstract}
We consider the defocusing energy-critical nonlinear Schr\"odinger equation with inverse-square potential  $iu_t = -\Delta u + a|x|^{-2}u + |u|^4u$ in three space dimensions. We prove global well-posedness and scattering for $a>-\frac14 +\frac1{25}$.  We also carry out the variational analysis needed to treat the focusing case.
\end{abstract}

\maketitle

\begin{center}
 \begin{minipage}{100mm}
   { \small {{\bf Key Words:}  Nonlinear Schr\"odinger equation;  scattering;  inverse-square potential;
   concentration compactness}
      {}
   }\\
    { \small {\bf AMS Classification:}
      {35P25,  35Q55, 47J35.}
      }
 \end{minipage}
 \end{center}
 

\section{Introduction}

We study the initial-value problem for the energy-critical defocusing nonlinear Schr\"odinger equation with an inverse-square potential in three spatial dimensions.  Before formulating this problem precisely, we first need to give a proper formulation of the associated linear problem. 

For $a\geq -\frac14$, the map
$$
Q(f) := \int_{\R^3} |\nabla f(x)|^2 + \tfrac{a}{|x|^2} |f(x)|^2\,dx
$$
defines a positive definite quadratic form on $C^\infty_c(\R^3\setminus\{0\})$.  Indeed, this can be verified by observing that
$$
Q(f) = \int_{\R^3} \bigl| \nabla f + \tfrac{\sigma x}{|x|^2} f \bigr|^2 \,dx \qtq{with} \sigma = \tfrac12 - \sqrt{\tfrac14+a}\,.
$$
We then define the operator
$$
\ld := -\Delta + \tfrac{a}{|x|^2}
$$
as the Friedrichs extension of the quadratic form $Q(f)$; see, for example, \cite[\S X.3]{RS2} for the general theory of such extensions and \cite[\S1.1]{KMVZZ} or \cite{KSWW} for more on this particular operator.  We choose the Friedrichs extension for physically motivated reasons:  (i) when $a=0$ it yields the usual Laplacian $-\Delta$ and (ii) the Friedrichs extension appears when one takes a scaling limit of more regular potentials.  For example,
$$
L_n:=-\Delta + \frac{a n^2}{1+n^2|x|^2} \longrightarrow  \ld \qtq{as} n\to\infty
$$
in strong resolvent sense (where we understand $L_n$ as having domain $H^2(\R^3)$).

The fact that $Q(f)\geq 0$ when $a \geq -\tfrac{1}{4}$, but not smaller, is one realization of the sharp Hardy inequality.  When $a<-\tfrac14$, all self-adjoint extensions of $-\Delta+a|x|^{-2}$ are unbounded below; we do not consider this case at all in what follows.  On the other hand,  the sharp Hardy inequality also shows that
\begin{equation}\label{E:equiv I}
Q(f) = \bigl\| \sqrt{\ld} f \bigr\|_{L^2_x}^2  \sim \bigl\| \nabla f \bigr\|_{L^2_x}^2 \qtq{for} a > -\tfrac14.
\end{equation}
Note that this is an assertion of the isomorphism of the Sobolev spaces $\dot H^1(\R^3)$ defined in terms of powers of $\ld$ and via the usual gradient.  The determination of the sharp range of parameters for which such an equivalence holds was the main result of the paper \cite{KMVZZ}.  For ready accessibility, we repeat this result as Theorem~\ref{pro:equivsobolev} below.

In this paper, we study the initial-value problem
\begin{align} \label{equ1.1}
\begin{cases}    (i\partial_t-\ld)u= |u|^4u,\quad (t,x)\in\R\times\R^3, \\
u(0,x)=u_0(x) \in \dot H^1_x(\R^3),
\end{cases}
\end{align}
for $u:\R_t^{ }\times\R_x^3\to \C$.  The necessity of interpreting $-\Delta + a|x|^{-2}$ as a self-adjoint operator $\ld$ becomes more apparent when \eqref{equ1.1} is rewritten as an integral equation:
\begin{align}\label{E:StrSol}
u(t) = e^{-it\ld} u_0 - i \int_0^t e^{-i(t-s)\ld} |u(s)|^4u(s)\,ds.
\end{align}
In particular, it gives meaning to $e^{-it\ld}$ via the Hilbert-space functional calculus.

By a \emph{solution} to \eqref{equ1.1} we will always mean a strong solution, that is, a function $u\in C_t^{ }\dot H^1_x \cap L^{10}_{t,\text{loc}}L^{10}_x$ that obeys \eqref{E:StrSol}.  In Section~\ref{SS:LWP}, we will prove that such solutions exist, at least locally in time, when $a>-\tfrac14+\frac1{25}$.  The restriction $a> -\frac14+\frac1{25}$ represents the limit of what can be done within the confines of the usual Strichartz methodology.  This breakdown does not originate in the failure of Strichartz estimates for the propagator $e^{-it\ld}$; indeed, the paper \cite{BPST1} shows that the full range of such estimates holds for $a>-\tfrac14$.  Rather, it stems from failures in Sobolev embedding.  We will explain this more fully in Section~\ref{SS:LWP}.

The main result of this paper, however, is to prove global well-posedness of the problem \eqref{equ1.1} and determine the asymptotic behavior of solutions as $t\to\pm\infty$.  A key ingredient in extending the local result to a global one is the conservation of energy.  More precisely, solutions to \eqref{equ1.1} preserve
\begin{equation}\label{energy}
 E(u(t)):=\int_{\R^3} \tfrac12|\nabla u(t,x)|^2 + \tfrac{a}{2|x|^2} |u(t,x)|^2 + \tfrac1{6} |u(t,x)|^{6}  \,dx.
\end{equation}
In view of \eqref{E:equiv I}, this guarantees that solutions remain bounded in $\dot H^1_x(\R^3)$ for all time.  By itself, this does not suffice to guarantee that solutions can be continued indefinitely.  The difficulty stems from the energy-criticality of the nonlinearity.  If $u(t,x)$ is a solution to \eqref{equ1.1}, then so is
$\lambda^{1/2} u (\lambda^2 t, \lambda x)$; moreover, both solutions have the same energy.  Correspondingly, one must fear that a solution may concentrate a positive amount of energy at a single point in finite time; the solution cannot then be continued past this time, at least, not in the class of strong solutions.

In the setting of nonlinear Schr\"odinger equations, the scourge of energy-criticality was overcome first in the setting of \eqref{equ1.1} with $a=0$.  This was ground-breaking work of Bourgain \cite{Bo99a} for radial (=spherically symmetric) data and of Colliander, Keel, Staffilani, Takaoka, and Tao, \cite{CKSTT07}, in the case of general data.  Their result reads as follows: 

\begin{theorem}\label{T:gopher} Given $u_0\in\dot H^1(\R^3)$, there is a unique strong solution $u$ to
\begin{align} \label{nls}
\begin{cases}    (i\partial_t+ \Delta)u= |u|^4u,\quad
(t,x)\in\R\times\R^3,
\\
u(0,x)=u_0(x).
\end{cases}
\end{align}
Moreover, the solution $u$ is global in time and obeys the estimate
\begin{equation}\label{Gophergoal}
\int_{\R}\int_{\R^3}|u(t,x)|^{10}\,dx\,dt\leq C( \|u_0\|_{\dot H^1_x} ).
\end{equation}
\end{theorem}

The space-time bound \eqref{Gophergoal} guarantees that solutions scatter, that is, are asymptotic to linear solutions as $t\to\pm\infty$ in the sense of \eqref{E:Scat}; see, for example, Corollary~1.2 in \cite{CKSTT07}.

The analogue of Theorem~\ref{T:gopher} has been known in general dimensions for some time \cite{RV,TaoRadial,Visanphd,Visan2007}.   In this paper, we focus on the new difficulties introduced by the presence of the inverse-square potential and so restrict ourselves to three spatial dimensions for concreteness.  We are aware of no new difficulties associated to the higher dimensional problem beyond those discussed in Section~\ref{SS:LWP} in connection with the stability theory.  Here is the main result of this paper:

\begin{theorem}\label{theorem} Fix  $a> -\frac14+\frac1{25}$. Given $u_0\in\dot H^1(\R^3)$, there is a unique global solution $u$ to \eqref{equ1.1} satisfying
\begin{equation}\label{goal}
\int_{\R}\int_{\R^3}|u(t,x)|^{10}\,dx\,dt\leq C( \|u_0\|_{\dot H^1_x} ).
\end{equation}
Moreover, the solution $u$ scatters, that is, there exist unique $u_\pm\in \dot H^1(\R^3)$ such that
\begin{equation}\label{E:Scat}
\lim_{t\to\pm\infty}\|u(t)-e^{-it\ld}u_\pm\|_{\dot H^1_x}=0.
\end{equation}
\end{theorem}

It is not feasible to prove Theorem~\ref{theorem} by parroting the arguments from \cite{CKSTT07} or the ones from \cite{KV:quintic}, which gives a new proof of Theorem~\ref{T:gopher} in light of the recent advances on dispersive equations at the critical regularity.  Among other things, those arguments rely on the $L_x^1\to L_x^\infty$ dispersive estimate for the propagator $e^{it\Delta}$ and the coercivity of the interaction Morawetz identity, both of which break down for negative values of the coupling constant $a$.  Nevertheless, we will be using Theorem~\ref{T:gopher} as a black-box, as we will explain below.  

The proof of Theorem~\ref{theorem} employs the induction on energy argument pioneered in \cite{Bo99a,CKSTT07} and subsequently reimagined in \cite{KM}.  Some of the ingredients that underlie the Kenig--Merle approach to induction on energy are the ideas of linear and nonlinear profile decompositions from \cite{BahouriGerard} and the notion of a minimal blowup solution from \cite{Keraani}; however, the potential of these ideas for addressing the well-posedness problem for NLS was first realized in \cite{KM}. For a pedagogical introduction to the Kenig--Merle argument illustrating how it may be applied to the case of the defocusing energy-critical NLS without potential, see, for example, \cite{KVnote}.

The argument proceeds by contradiction.  Assuming that Theorem~\ref{theorem} were to fail, one first demonstrates the existence of a minimal counterexample, that is, a solution to \eqref{equ1.1} that has infinite space-time norm and has minimal energy among all such solutions.  The (concentration) compactness argument used to show this existence shows more, namely, that such a minimal blowup solution must be almost periodic (cf. Theorem~\ref{T:mmbs}).  The second half of the argument is to use monotonicity formulae and/or conservation laws to rule out the existence of such solutions.  The almost periodicity is essential here; it provides both a spatial center and a length scale that are intrinsic to the solution at hand.  These are essential for using conservation laws and/or monotonicity formulae that are not translation invariant and/or that do not respect the scaling of the equation.

The proof of the existence of a minimal blowup solution hinges on the construction of a nonlinear profile decomposition, which is applied to a sequence of solutions $u_n$ witnessing the supposed failure of Theorem~\ref{theorem}.  This decomposition says that (after passing to a subsequence and up to a negligible error) the solutions with initial data $u_n(0)$ can be asymptotically expressed as a linear combination of the nonlinear evolutions of a fixed collection of profiles; the profiles are, however, modified (in an $n$-dependent manner) by the symmetries of the equation, namely, space and time translations, as well as scaling. 

Here we see a new obstacle introduced by the presence of a non-zero potential; it breaks the space translation symmetry.  Ultimately, this means that profiles living increasingly far from the origin relative to their intrinsic scale cannot be treated inductively --- they cannot be modeled by a single solution of \eqref{equ1.1} up to symmetries of the equation.  On the other hand, we may expect that the solution stemming from such data is little affected by the potential and can be approximated by a solution to \eqref{nls}, thus inheriting the space-time bounds guaranteed by Theorem~\ref{T:gopher}.  This heuristic is realized in Theorem~\ref{T:embed}, whose proof fills Section~\ref{S:4}.  Both this result and the proof that such profiles decouple from the other nonlinear profiles rely on various convergence statements for linear operators proved in Section~\ref{S:LPD}.  The necessity of proving such convergence statements stems from the absence of $L_x^p\to L_x^{p'}$ dispersive estimates for the propagator $e^{-it\ld}$ when $a<0$; see \cite{PSS1}.  Indeed, one should view \eqref{convg op4} as a substitute for the dispersive estimate in this setting (cf. \cite[Theorem~4.1]{KVZ12}); this is essential for the construction of the minimal blowup solution.  

This paper is not the first to implement the credo that to treat critical dispersive equations with broken symmetries, one must first solve the limiting problems where the symmetries are restored.  Previous examples include \cite{IonPaus1,IonPaus,IPS:H3,Jao1,Jao2,KKSV:gKdV,KOPV,KSV:2DKG,KVZ12,PausTzW}.  This is the first model to be studied, however, which retains scale invariance, but loses space translation invariance.

The credo stated above also explains one of our motivations for studying this particular problem; by its precepts, the study of the energy-critical NLS with a potential such as $V(x)=1/(1+|x|^2)$ relies on a satisfactory treatment of the problem \eqref{equ1.1}, which appears as a scaling limit.  We were further drawn to this problem by the frequency with which the inverse-square potential appears in physical contexts.  Several such instances are discussed in the mathematical papers \cite{BPST1,BPST2,KSWW,VZ,ZZ}, ranging from combustion theory to the Dirac equation with Coulomb potential, and to the study of perturbations of classic space-time metrics such as Schwarzschild and Reissner--Nordstr\"om.

The construction of minimal blowup solutions to \eqref{equ1.1}, under the assumption that Theorem~\ref{theorem} fails, is completed in Section~\ref{S:5}.  Because profiles far from the origin can be discounted by piggybacking on Theorem~\ref{T:gopher} in the manner sketched above, we are guaranteed that such a minimal blowup solution remains close to the origin (relative to its intrinsic scale) for all time.  However, the argument yields no restrictions on the behavior of this characteristic scale $1/N(t)$ beyond the local constancy property inherited from the local well-posedness theory (cf. \cite[Lemma 5.18]{KVnote}). 

By simple rescaling and time-translation arguments, one can extract from any minimal blowup solution a minimal blowup solution that has $N(t)\geq 1$ on half of its evolution, say, for $t\geq 0$.  Indeed, further combinatorial arguments allow one to further confine the behavior of $N(t)$ as shown in \cite[Theorem~1.17]{KV5}; however, the fact that the solution remains centered near the origin makes such refinement unnecessary in this paper.

In order to preclude minimal blowup solutions (and thereby prove Theorem~\ref{theorem}) we consider two cases: (i) the solution blows up in finite time in the future $T^*<\infty$, or (ii) the solution is global in the future $T^*=\infty$.  We exclude the former scenario by using the conservation of mass
\begin{equation}\label{mass}
M(u)=\int_{\mathbb{R}^3}|u(t,x)|^2\, dx;
\end{equation}
see Theorem~\ref{T:7.2}.  By minimality of the solution, finite-time blowup is accompanied by movement of the solution to high frequencies, leaving no mass on any ball of fixed radius.  As mass transportation is controlled by the energy, which is bounded, this absence of mass can be propagated backward in time.  Sending the radius to infinity and then sampling time to $T^*$, we deduce that the minimal blowup solution has zero mass and so is zero itself.  This contradicts the presumption that the solution blows up.

To preclude the case of a solution that exists globally in the future, we use a truncated virial identity.  For a defocussing nonlinearity, it would be more usual to use the interaction Morawetz identity.  The big advantage of the interaction Morawetz identity is that it is insensitive to the spatial location of the solution.   Unfortunately, when $a<0$, the potential produces terms with an unfavorable sign and consequently, we are unable to prove coercivity for the full range of $a$. On the other hand, for the problem treated in this paper we are able to constrain the motion to remain near the origin.  This makes it possible to use a virial argument, for which coercivity follows from the sharp Hardy inequality.  Of course, it is necessary to truncate the usual virial identity in space (it has unbounded coefficients); however, the almost periodicity of the solution guarantees that such a truncation does not destroy the coercivity.  See Theorem~\ref{T:7.1} for details.

Given the use of the virial identity, it is natural to ask for an analogue of Theorem~\ref{theorem} in the focusing case.  In the focusing case, one expects both solitons and blowup, so scattering cannot hold for all finite-energy initial data.  The natural conjecture in this setting is that scattering holds for data below the threshold given by the least energy soliton.  In Section~\ref{S:7} we carry out the requisite variational analysis to determine the natural conjectural threshold for the focusing problem in the presence of an inverse-square potential.  When $a<0$, this threshold is dictated by a soliton centered on the origin.  When $a>0$, the threshold coincides with that associated to the case $a=0$, since solutions occurring in the absence of a potential can be embedded far from the origin.  As evidence of the correctness of these thresholds, we show that the virial identity is coercive below these thresholds and show the viability of the concavity argument to prove blowup above the threshold; see Corollary~\ref{C: trap} and Proposition~\ref{P:blowup}.

The concentration compactness arguments used in this paper carry over to the focusing case with little difficulty.  Thus, to prove the natural threshold conjecture for the focusing equation with inverse-square potential we are missing one last ingredient, namely, an analogue of Theorem~\ref{T:gopher} in the focusing case.  However, the full threshold conjecture for the focusing energy-critical NLS in three dimensions remains open.  It has been resolved in dimensions four and higher \cite{Dod,KV5}, but in three dimensions, it is currently only known for radial data \cite{KM}.

\subsection*{Acknowledgments:}  R. Killip was supported by grant DMS-1265868 from the NSF. C.~Miao was supported by the NSFC under grants
11171033 and 11231006. M. Visan was supported by NSF grant DMS-1500707. J. Zhang was supported by the Excellent Young Scholars
Research Fund of the Beijing Institute of Technology, Beijing Natural Science Foundation (1144014), and by NSFC grant No.~11401024.
J. Zheng was partly supported by the ERC grant SCAPDE.




\section{Preliminaries}

We start by introducing some of the notation used throughout this article.  If $X, Y$ are non-negative quantities, we write $X\lesssim Y $ or $X=O(Y)$ when $X\leq CY$ for some $C$.  We will use the shorthand $A\vee B:=\max\{A, B\}$ and $\langle x\rangle:=\sqrt{1+|x|^2}.$

For $1< r < \infty$, we write $\dot H^{1,r}_a(\R^3)$ and $ H^{1,r}_a(\R^3)$ for the homogeneous and inhomogeneous Sobolev spaces associated with $\ld$, respectively, which have norms
$$
\|f\|_{\dot H^{1,r}_a(\R^3)}= \|\sqrt{\ld} f\|_{L^r(\R^3)} \qtq{and} \|f\|_{H^{1,r}_a(\R^3)}= \|\sqrt{1+ \ld} f\|_{L^r(\R^3)}.
$$
When $r=2$, we simply write $\dot H^{1}_a(\R^3)=\dot H^{1,2}_a(\R^3)$ and $H^{1}_a(\R^3)=H^{1,2}_a(\R^3)$.

Some of the results we will use admit generalizations to dimensions $d\geq 3$.   Note that the operator $\ld$ is positive precisely for $a\geq -(\frac{d-2}2)^2$; this is the sharp form of Hardy's inequality.  Many results have clearer formulations when written in terms of the parameter
$$
\sigma:=\tfrac{d-2}2-\bigr[\bigl(\tfrac{d-2}2\bigr)^2+ a\bigr]^{1/2},
$$
rather than the coupling constant $a$.

Estimates on the heat kernel associated to the operator $\ld$ were found by Liskevich--Sobol \cite{LS} and Milman--Semenov \cite{MS}.

\begin{theorem}[Heat kernel bounds]\label{T:heat} Fix $d\geq 3$ and $a\geq-(\frac{d-2}2)^2$.  There exist positive constants $C_1,C_2$ and $c_1,c_2$ such that for all
$t>0$ and all $x,y\in \R^d\setminus\{0\}$,
\begin{equation*}
C_1\bigl( 1\vee \tfrac{\sqrt t}{|x|} \bigr)^\sigma \bigl( 1\vee \tfrac{\sqrt t}{|y|} \bigr)^\sigma t^{-\frac d2}e^{-\frac{|x-y|^2}{c_1t}}
\leq e^{-t\mathcal L_a}(x,y)\leq C_2\bigl( 1\vee \tfrac{\sqrt t}{|x|} \bigr)^\sigma \bigl( 1\vee \tfrac{\sqrt t}{|y|} \bigr)^\sigma t^{-\frac d2}e^{-\frac{|x-y|^2}{c_2t}}.
\end{equation*}
\end{theorem}

These estimates formed the starting point of the analysis in \cite{KMVZZ}, which develops a number of basic harmonic analytic tools that we will use in this paper.  Foremost among these is the following theorem, which tells us when Sobolev spaces defined through powers of $\ld$ coincide with the traditional Sobolev spaces. 

\begin{theorem}[Equivalence of Sobolev spaces]\label{pro:equivsobolev} Fix $d\geq 3$, $a\geq-(\tfrac{d-2}2)^2$, and $0<s<2$. If $1<p<\infty$ satisfies
$\frac{s+\sigma}d<\frac1p<\min\{1, \frac{d-\sigma}d\}$, then
\begin{equation}\label{RieszT}
\big\|(-\Delta)^\frac{s}2f\big\|_{L^p}\lesssim_{d,p,s}\big\|\ld^\frac{s}2 f\big\|_{L^p} \quad \text{for all}~ f\in C_c^\infty(\R^d\setminus\{0\}).
\end{equation}
If $\max\{\frac sd, \frac\sigma d\}<\frac1p<\min\{1,\frac{d-\sigma}d\}$, which ensures already that $1<p<\infty$, then
\begin{equation}\label{RieszT'}
\big\|\ld^\frac{s}2f\big\|_{L^p}\lesssim_{d,p,s}\big\|(-\Delta)^\frac{s}2f\big\|_{L^p}\quad\text{for all}~ f\in C_c^\infty(\R^d\setminus\{0\}).
\end{equation}
\end{theorem}

Our main result, Theorem~\ref{theorem}, is restricted to the range $a>-\frac14+\frac1{25}$.  As mentioned in the Introduction, this is needed to prove even local well-posedness via the usual methodology.  The restriction  $a>-\frac14+\frac1{25}$ guarantees that $\sigma < \frac3{10}$ and consequently,
$$
\big\|\nabla f\big\|_{L^p(\R^3)} \sim \big\| \sqrt{\ld} f\big\|_{L^p(\R^3)} \qtq{for all}  \tfrac65 \leq p \leq \tfrac{30}{13}.
$$

The $L^p$-product rule for fractional derivatives in Euclidean spaces was first proved by Christ and Weinstein \cite{ChW:fractional chain rule}.  Combining their result with Theorem~\ref{pro:equivsobolev} yields an analogous statement for the operator $\ld$.  We record here only the particular cases that will be needed in this paper.

\begin{lemma}[Fractional product rule]\label{L:Leibnitz}
Fix $a>-\frac14+\frac1{25}$.  Then for all $f, g\in C_c^{\infty}(\R^3\setminus\{0\})$ we have
\begin{align*}
\| \sqrt{\ld}(fg)\|_{L^p(\R^3)} \lesssim \| \sqrt{\ld} f\|_{L^{p_1}(\R^3)}\|g\|_{L^{p_2}(\R^3)}+\|f\|_{L^{q_1}(\R^3)}\| \sqrt{\ld} g\|_{L^{q_2}(\R^3)},
\end{align*}
for any exponents satisfying $\frac65\leq p, p_1, p_2, q_1, q_2\leq \frac{30}{13}$ and $\frac1p=\frac1{p_1}+\frac1{p_2}=\frac1{q_1}+\frac1{q_2}$.
\end{lemma}

Let $\phi:[0,\infty)\to[0,1]$ be a smooth positive function obeying
\begin{align*}
\phi(\lambda)=1 \qtq{for} 0\le\lambda\le 1 \qtq{and} \phi(\lambda)=0\qtq{for}\lambda\ge 2.
\end{align*}
For each dyadic number $N\in 2^\Z$, we define
\begin{align*}
\phi_N(\lambda):=\phi(\lambda/N) \qtq{and} \psi_N(\lambda):=\phi_N(\lambda)-\phi_{N/2}(\lambda).
\end{align*}
It is clear that $\{\psi_N(\lambda)\}_{N\in 2^{\Z}} $ is a partition of unity for $(0,\infty)$.  We define the Littlewood--Paley projections via
\begin{gather*}
f_{\leq N}:=P_{\le N}^a f:= \phi_N\bigl(\sqrt{\ld}\,\bigr), \qquad f_N:=P^a_N f:=\psi_N(\sqrt{\ld}\,\bigr),\\
\text{and} \qquad \ f_{>N}:= P^a_{>N} f:=(I-P^a_{\le N})f.
\end{gather*}
We will also make use of Littlewood--Paley projections defined via the heat kernel:
\begin{align*}
\tilde P_N^a:=e^{-\mathcal L_a/N^2}-e^{-4\mathcal L_a/N^2}.
\end{align*}

Next we recall the following Bernstein and square function estimates, which were proved in \cite{KMVZZ}:

\begin{lemma}[Bernstein estimates]\label{L:Bernie}
For $1<p\leq q\le \infty$ when $a\geq0$ or $r_0<p\leq q<r_0':=\frac{d}\sigma$ when $-(\frac{d-2}2)^2\leq a<0$, the following hold:

\noindent$(1)$ The operators $P^a_{\le N}$, $P^a_{N}$, and $\tilde P_N^a$ are bounded on $L^p$.

\noindent$(2)$ The operators $P^a_{\le N}$, $P^a_{N}$, and $\tilde P_N^a$ map $L^p$ to $L^q$ with norm $O(N^{\frac dp-\frac dq})$.

\noindent$(3)$ For any $s\in \R$,
$$N^s\|P^a_N f\|_{L_x^p}  \sim \bigl\|(\ld)^{\frac s2}P^a_N f\bigr\|_{L_x^p}\qtq{and} N^s\|\tilde P^a_N f\|_{L_x^p}  \sim \bigl\|(\ld)^{\frac s2}\tilde P^a_N f\bigr\|_{L_x^p}.$$
\end{lemma}

\begin{lemma}[Square function]\label{T:sq}
Fix $0\leq s<2$.  For $1<p<\infty$ when $a\geq0$ or $r_0<p<r_0':=\frac d{\sigma}$ when $-(\frac{d-2}2)^2\leq a<0$ and any $f\in C_c^{\infty}(\R^d\setminus\{0\})$,
\begin{align*}
\biggl\|\biggl(\sum_{N\in2^\Z} N^{2s}| P^a_N f|^2\biggr)^{\!\!\frac 12}\biggr\|_{L^p} \sim \|\ld^{\frac s2}f\|_{L^p}.
\end{align*}
\end{lemma}

In order to prove decoupling of the potential energy in our nonlinear profile decomposition, we will use the following refined Fatou lemma due to Brezis and Lieb \cite{BrezisLieb}.

\begin{lemma}[Refined Fatou, \cite{BrezisLieb}]\label{reflemms} Fix $1\leq p<\infty$ and let $\{f_n\}$ be a bounded sequence in $L^p(\R^d)$.  If
$f_n\to f$ almost everywhere, then
$$\int_{\R^3}\big||f_n|^p-|f_n-f|^p-|f|^p\big|\,dx\to 0.$$
In particular,
$\|f_n\|_{L_x^p}^p-\|f_n-f\|_{L_x^p}^p\to\|f\|_{L_x^p}^p.$
\end{lemma}

Strichartz estimates for the propagator $e^{-it\ld}$ in $\R^3$ were proved by Burq, Planchon, Stalker, and Tahvildar-Zadeh in \cite{BPST1}.  Combining these with the Christ--Kiselev Lemma \cite{CK}, we obtain the following Strichartz estimates:

\begin{theorem}[Strichartz estimates, \cite{BPST1}] Fix $a>-\tfrac14$.  The solution $u$ to
$$
i\partial_t u = \ld u + F
$$
on an interval $I\ni t_0$ obeys
\begin{equation}\label{str1}
\|u\|_{L_t^qL_x^r(I\times\R^3)} \lesssim \|u(t_0)\|_{L^2(\R^3)}+\|F \|_{L_t^{\tilde q'}L_x^{\tilde r'}(I\times\R^3)},
\end{equation}
whenever
$\tfrac2q+\tfrac3r = \tfrac2{\tilde q}+\tfrac3{\tilde r} = \tfrac32$, $2\leq q,\tilde q \leq \infty$, and $q\neq \tilde q$.
\end{theorem}

We would like to point out the loss of the double endpoint Strichartz estimate in the theorem above.  This can be recovered via the argument in \cite{KeelTao} for $a\geq0$, since $L^1_x\to L_x^\infty$ dispersive estimates were proved in this case; see \cite[Theorem~1.11]{Fanelli}.  For sufficiently small $a<0$, one can also recover the double endpoint Strichartz estimate via the argument in \cite{BPST1}, by treating the potential term as a perturbation. 

The application of this theorem is simplified by the introduction of the spaces
\begin{gather*}
S^0(I):= L_t^2L_x^6 (I\times\R^3)\cap L_t^\infty L_x^2(I\times\R^3), \\
\dot S^1(I):=\{u:I\times\R^3\to \C:\, \nabla u\in S^0(I)\},
\end{gather*}
and the dual Strichartz space $N^0(I):=L_t^2L_x^{6/5} (I\times\R^3) +  L_t^1 L_x^2(I\times\R^3)$.

We record next a local smoothing result for the linear propagator $e^{-it\ld}$.  Its only application in this article will be to prove Corollary~\ref{C:Keraani3.7}.

\begin{lemma}[Local smoothing]\label{L:local smoothing}
Fix $a> -\frac14$ and let $u=e^{-it\ld} u_0$. Then
\begin{align}
\int_\R \int_{\R^3}  \frac{|\nabla u(t,x)|^2}{ R \langle R^{-1} x\rangle^{3}}  + \frac{|u(t,x)|^2}{ R |x|^2} \,dx\,dt
	&\lesssim \| u_0 \|_{L^2_x} \|\nabla u_0 \|_{L^2_x} + R^{-1} \| u_0 \|_{L^2_x}^2, \label{LS:at 0}\\
\int_\R \int_{|x-z|\leq R}  \; \tfrac{1}{R} |\nabla u(t,x)|^2 \,dx\,dt 	&\lesssim \| u_0 \|_{L^2_x} \|\nabla u_0 \|_{L^2_x} + R^{-1} \| u_0 \|_{L^2_x}^2, \label{LS:at z}
\end{align}
uniformly for $z\in \R^3$ and $R>0$.
\end{lemma}

\begin{proof}
Theorem~1 of \cite{BPST1} shows that 
\begin{equation}\label{KatoS}
\int_\R \int_{\R^3}  \frac{|u(t,x)|^2}{ R |x|^2} \,dx\,dt \lesssim  R^{-1} \| u_0 \|_{L^2_x}^2.
\end{equation}

To complete the proof of \eqref{LS:at 0}, we will use a weighted momentum identity analogous to the virial and Morawetz identities.  To this end, let
$$
F(t) := \int_{\R^3} 2\Im( \bar u \partial_j u) \, b_j(x)\, dx \qtq{with} b_j(x) = \tfrac{x_j/R}{\langle x/R\rangle},
$$
where repeated indices are summed.  Clearly
\begin{equation}\label{F bnd}
|F(t)|\leq 2 \| u_0 \|_{L^2_x} \|\nabla u_0 \|_{L^2_x},
\end{equation}
while direct computation shows
$$
\partial_t F(t)= \int_{\R^3} -b_{jjkk}|u|^2  + 4 b_{jk}\bar u_j u_k  +4a b_j x_j \tfrac{|u|^2}{|x|^4} \, dx,
$$
where additional subscripts indicate differentiation and repeated indices are summed.

For our particular choice of $b$, we have
$$
-b_{jjkk}(x) = \tfrac{15}{R^3\langle x/R\rangle^7}\geq 0, \quad 4a b_j(x) \tfrac{x_j}{|x|^4} \geq - \tfrac{4|a|}{R|x|^2},
$$
and
$$ 
	b_{jk}(x) = \tfrac{\delta_{jk}}{R\langle x/R\rangle}  -\tfrac{x_j x_k}{R^3\langle x/R\rangle^3} \geq  \tfrac{\delta_{jk}}{R\langle x/R\rangle^3}
$$
in the sense of matrices.  The estimate \eqref{LS:at 0} now follows by applying the fundamental theorem of calculus and using \eqref{KatoS} to bound the potential term, which has an unfavorable sign when $a<0$.

We now turn our attention to \eqref{LS:at z}.  When $|z|\leq 2R$ this estimate follows from \eqref{LS:at 0}; thus we need only consider the case $|z|\geq 2R$.  We use the weighted momentum identity once more, but this time choosing
$$
b_j(x) = \tfrac{(x_j-z_j)/R}{\langle (x-z)/R \rangle} \phi(x/R)
$$
where $\phi(y)\geq 0$ is a smooth function vanishing when $|y|\leq \frac12$ and obeying $\phi(y)=1$ when $|y|\geq 1$.  For this choice of weight, we again have \eqref{F bnd}; moreover,
$$
-b_{jjkk}(x) \gtrsim_\phi - \tfrac{1}{R |x|^2}, \quad 4a b_j(x) \tfrac{x_j}{|x|^4} \geq - \tfrac{8|a|}{R|x|^2},
$$
and
$$
b_{jk}(x)  \geq  \tfrac{\delta_{jk}}{R\langle (x-z)/R\rangle^3} \phi(x/R) - \tfrac{ 1 }{R} |[\nabla\phi](x/R)|.
$$
The estimate \eqref{LS:at z} now follows by applying the fundamental theorem of calculus, using \eqref{LS:at 0} to bound those terms with an unfavorable sign.
\end{proof}

The following corollary will be used to prove a Palais--Smale condition for minimizing sequences of blowup solutions to \eqref{equ1.1}.  Our models for this result are \cite[Corollary~4.15]{KVnote} and  \cite[Corollary~2.14]{KVZ12}.

\begin{corollary}\label{C:Keraani3.7}
Fix $a>-\frac14+\frac1{25}$ and let $w_0\in \dot H^1_a(\R^3)$.  Then
\begin{align*}
\| \nabla e^{-it\ld} w_0\|_{L_t^5L_x^\frac{15}8([\tau-T,\tau+T]\times\{|x-z|\leq R\})}
&\lesssim T^{\frac{29}{320}} R^{\frac{51}{160}} \| e^{-it\ld} w_0 \|_{L^{10}_{t,x}(\R\times\R^3)}^{\frac1{32}}\| w_0 \|_{\dot H^1_x}^{\frac{31}{32}}\\
&\ +T^{\frac{29}{280}}R^{\frac{41}{140}}\| e^{-it\ld} w_0 \|_{L^{10}_{t,x}(\R\times\R^3)}^{\frac1{28}} \| w_0 \|_{\dot H^1_x}^{\frac{27}{28}},
\end{align*}
uniformly in $w_0$ and the parameters $R,T > 0$, $\tau\in\R$, and  $z\in\R^3$.
\end{corollary}

\begin{proof}
Replacing $w_0$ by $e^{-i\tau\ld} w_0$, we see that it suffices to treat the case $\tau=0$.

By H\"older's inequality, Theorem~\ref{pro:equivsobolev}, and Strichartz,
\begin{align*}
\| \nabla e^{-it\ld}&  w_0\|_{L_t^5L_x^\frac{15}8([-T,T]\times\{|x-z|\leq R\})}\\
&\lesssim R^\frac14 \|\nabla e^{-it\ld} w_0\|_{L_t^{10}L_x^\frac{30}{13}([-T,T]\times\R^3)}^{\frac34}\| \nabla e^{-it\ld} w_0\|_{L^2_{t,x}([-T,T]\times\{|x-z|\leq R\})}^{\frac14} \\
&\lesssim R^\frac14\|w_0\|_{\dot H^1(\R^3)}^{\frac34}\| \nabla e^{-it\ld} w_0\|_{L^2_{t,x}([-T,T]\times\{|x-z|\leq R\})}^{\frac14}.
\end{align*}

To continue, we consider small and high frequencies separately.  Fix $N\in 2^{\Z}$.  Using Theorem~\ref{pro:equivsobolev}, H\"older, Bernstein, and Strichartz, we estimate the low frequencies as follows:
\begin{align*}
\bigl\| \nabla e^{-it\ld} &P_{< N}^a  w_0 \bigr\|_{L^2_{t,x}([-T,T]\times\{|x-z|\leq R\})}\\
&\lesssim T^{\frac{29}{60}} R^{\frac15} \bigl\| \nabla e^{-it\ld} P_{< N}^a w_0 \bigr\|_{L^{60}_tL^{\frac{30}{13}}_x(\R\times\R^3)} \\
&\lesssim T^{\frac{29}{60}} R^{\frac15} N^{\frac16} \| (\ld)^{\frac{5}{12}} e^{-it\ld} P_{< N}^a w_0 \|_{L^{60}_t L^{\frac{30}{13}}_x(\R\times\R^3)} \\
&\lesssim T^{\frac{29}{60}} R^{\frac15} N^{\frac16} \| e^{-it\ld}w_0 \|_{L^{10}_{t,x}(\R\times\R^3)}^{\frac16} \| (\ld)^{\frac12} e^{-it\ld} w_0 \|_{L^{\infty}_t L^{2}_x(\R\times\R^3)}^{\frac56} \\
&\lesssim T^{\frac{29}{60}} R^{\frac15} N^{\frac16} \| e^{-it\ld}w_0 \|_{L^{10}_{t,x}(\R\times\R^3)}^{\frac16} \| w_0 \|_{\dot H^1(\R^3)}^{\frac56}.
\end{align*}
To estimate the high frequencies, we use Lemma~\ref{L:local smoothing} and Bernstein:
\begin{align*}
\bigl\| \nabla e^{-it\ld} P_{\geq N}^a w_0\bigr\|_{L^2_{t,x}([-T,T]\times\{|x-z|\leq R\})}^2
&\lesssim R \| P_{\geq N}^a w_0 \|_{L^2_x}  \| \nabla P_{\geq N}^a w_0 \|_{L^2_x}+ \|  P_{\geq N}^a w_0 \|_{L^2_x}^2\\
&\lesssim  \big(R N^{-1}+N^{-2}\big) \| w_0 \|_{\dot H^1(\R^3)}^2.
\end{align*}

The claim follows by optimizing in the choice of $N$.
\end{proof}

\subsection{Local well-posedness and stability}\label{SS:LWP}

First, we show local well-posedness in the inhomogeneous space $H^1_a(\R^3)$. Local well-posedness in the larger space $\dot H^1_a(\R^3)$ then follows as an application of the stability theory Theorem~\ref{thm:stability}.  In particular, one should note that none of the constants appearing in Proposition~\ref{P:LWP} make any reference to the mass of the initial data; the assumption of finite mass is simply a crutch employed in the contraction mapping argument.

\begin{proposition}\label{P:LWP} Fix $a>-\frac14+\frac1{25}$.  Given $A\geq 0$, there exists $\eta=\eta(A)$ so that the following holds:
Suppose $u_0\in H^1_a(\R^3)$ obeys
\begin{equation}\label{E:LWP hyp}
\| \sqrt{\ld} u_0 \bigr\|_{L^2(\R^3)} \leq A \qtq{and}  \bigl\| e^{it\ld} u_0 \bigr\|_{L^{10}_{t,x}(I\times\R^3)} \leq \eta
\end{equation}
for some time interval $I\ni 0$.  Then there is a unique strong solution $u$ to \eqref{equ1.1} on the time interval $I$ such that
\begin{equation}\label{E:LWP conc}
\bigl\| u \bigr\|_{L^{10}_{t,x}(I\times\R^3) } \lesssim \eta
\qtq{and} 
\bigl\| \sqrt{\ld} \; u \bigr\|_{C_t L^2_x  \cap L_t^{10} L_x^{\frac{30}{13}} (I\times\R^3) } \lesssim A.
\end{equation}
\end{proposition}

\begin{proof}
Throughout the proof, all space-time norms will be on $I\times\R^3$.

It suffices to show that
$$
\Phi:u \mapsto  e^{-it\ld} u_0  - i \int_0^t e^{-i(t-s)\ld}  |u(s)|^4 u(s)\,ds
$$
is a contraction on the (complete) space
\begin{align*}
B:=\Bigl\{ u\in C_t^{ } H^{1}_a  \cap L_t^{10} H^{1,\frac{30}{13}}_a(I\times\R^3):\   & \bigl\| \sqrt{\ld} \; u \bigr\|_{L_t^{10} L_x^{\frac{30}{13}}} \leq C A
 \ \text{and}\  \bigl\| u \bigr\|_{L^{10}_{t,x}} \leq 2\eta \Bigr\}
\end{align*}
endowed with the metric
$$
d(u,v):= \| u -v \|_{L_t^{10} L_x^{\frac{30}{13}}}.
$$
The constant $C$ depends only on the dimension and $a$ and it reflects various constants in the Strichartz and Sobolev embedding inequalities.

By the Strichartz inequality, \eqref{E:LWP hyp}, and Lemma~\ref{L:Leibnitz}, for $u\in B$ we have
\begin{align*}
\bigl\|  \sqrt{\ld} \; \Phi (u) \bigr\|_{C_t^{ } L^2_x  \cap L_t^{10} L_x^{\frac{30}{13}}} &\lesssim  A + \bigl\| \sqrt{\ld}\; (|u|^4u) \bigr\|_{L_t^{2} L_x^\frac65}
	\lesssim A + \eta^4 \bigl\|  \sqrt{\ld} \; u \bigr\|_{L_t^{10} L_x^\frac{30}{13}}.
\end{align*}
Thus choosing $\eta$ sufficiently small we obtain that
\begin{equation}\label{self1}
\bigl\|  \sqrt{\ld} \; \Phi (u) \bigr\|_{C_t^{ } L^2_x  \cap L_t^{10} L_x^{\frac{30}{13}}} \leq  C A.
\end{equation}
Similarly, for $u\in B$,
\begin{align}\label{self2}
\bigl\| \Phi (u) \bigr\|_{C_t^{ } L^2_x  \cap L_t^{10} L_x^{\frac{30}{13}}} & \lesssim  \|u_0\|_{L^2_x} + \bigl\| |u|^4u \bigr\|_{L_t^{2} L_x^\frac65}
	\lesssim  \|u_0\|_{L^2_x} + \eta^4 \bigl\| u \bigr\|_{L_t^{10} L_x^{\frac{30}{13}}} <\infty.
\end{align}
Proceeding once more in a parallel manner shows
\begin{align*}
\bigl\| \Phi (u) \bigr\|_{L^{10}_{t,x}} \leq \eta + C' \bigl\| u\bigr\|_{L^{10}_{t,x}}^4 \bigl\| \sqrt{\ld}\, u \bigr\|_{L_t^{10} L_x^\frac{30}{13}}
	\leq \eta +  C'(2\eta)^4 2A \leq 2\eta,
\end{align*}
provided $\eta$ is chosen sufficiently small.  This, \eqref{self1}, and \eqref{self2} show that $\Phi$ does indeed map $B$ into itself.

To see that $\Phi$ is a contraction, we argue analogously:
\begin{align*}
\| \Phi(u) -\Phi(v) \|_{L_t^{10} L_x^{\frac{30}{13}}}&\lesssim \bigl\| |u|^4u- |v|^4v\bigr\|_{L_t^2 L_x^\frac65}\\
&\lesssim \| u - v \|_{L_t^{10} L_x^{\frac{30}{13}}} \Bigl( \|u\|_{L^{10}_{t,x}} + \|v\|_{L^{10}_{t,x}} \Bigr)^4,
\end{align*}
which shows that for $u,v\in B$, one has
$
d(\Phi(u),\Phi(v)) \leq \tfrac12 d(u,v)
$
provided $\eta$ is chosen sufficiently small.
\end{proof}

The condition $a>-\frac14+\frac1{25}$ expresses the limit of what can be achieved by the method presented above.  In order to estimate the nonlinearity in a dual Strichartz space, one must be able to place the solution in a space $L^q_t L^r_x$ with $q\leq 10$.  By scaling this then requires $r\geq 10$.  On the other hand, if $a\leq -\frac14+\frac1{25}$ then the Sobolev embedding $\dot H_a^{1,\frac{3r}{r+3}}\hookrightarrow L^r_x$ breaks down for all $r\geq 10$.

We turn now to the formulation of the stability theory.  As noted above, it allows us to remove the finite-mass assumption from Proposition~\ref{P:LWP}.  It also plays an essential role in the implementation of the induction on energy argument employed in this paper; in particular, it shows that solutions to \eqref{equ1.1} whose initial data are the sum of two parts that live at vastly different length scales can be approximated by the sum of the corresponding solutions to \eqref{equ1.1}.

\begin{theorem}[Stability]\label{thm:stability}
Fix $a>-\frac14+\frac1{25}.$  Let $I$ be a compact time interval and let $\tilde u$ be an approximate solution to \eqref{equ1.1} on $I\times \R^3$ in the sense that
$$
(i\partial_t -\ld) \tilde u=|\tilde u|^4\tilde u + e
$$
for some function $e$.  For some positive constants $E$ and $L$, assume that
\begin{align*}
\|\tilde u\|_{L_t^\infty \dot H_a^1(I\times \R^3)}\le E
\quad{and}\quad \|\tilde u\|_{L_{t,x}^{10}(I\times \R^3)} \le L.
\end{align*}
Let $t_0 \in I$ and let
$u_0\in \dot H^1_a(\R^3)$.  Assume the smallness
conditions
\begin{align*}
\|u_0-\tilde u(t_0)\|_{\dot H^1_a(\R^3)} +\bigl\|\sqrt{\ld}\; e\bigr\|_{N^0(I)}&\le\eps
\end{align*}
for some $0<\eps<\eps_1=\eps_1(E,L)$. Then, there exists a unique strong solution $u:I\times\R^3\mapsto \C$ to \eqref{equ1.1} with initial data $u_0$ at time $t=t_0$ satisfying
\begin{align*}
\bigl\|\sqrt{\ld}\;  (u-\tilde u)\bigr\|_{S^0(I)} &\leq C(E,L)\, \eps\\
\bigl\|\sqrt{\ld}\;  u\bigr\|_{S^0(I)} &\leq C(E,L).
\end{align*}
\end{theorem}

\begin{proof}
The proof of this theorem follows the general outline in \cite{CKSTT07,RV,TaoVisan}, exploiting the same spaces used in the proof of Proposition~\ref{P:LWP}.   As there, Lemma~\ref{L:Leibnitz} provides the key input for differentiating the nonlinearity with respect to the operator $\sqrt\ld$.
\end{proof}

Employing the arguments used to derive a stability theory for the energy-critical NLS without a potential and relying on Theorem~\ref{pro:equivsobolev} to differentiate the nonlinearity with respect to the operator $\sqrt\ld$, we can derive the analogue of Theorem~\ref{thm:stability} in dimensions $d\in\{4,5,6\}$, as long as $a>-(\frac{d-2}2)^2+ (\frac{d-2}{d+2})^2$.  A stability result for $d\geq 7$ is known in the setting $a=0$; see \cite{KVnote,TaoVisan}.  The proof relies on a fractional chain rule for H\"older continuous functions and exotic Strichartz estimates.  Theorem~\ref{pro:equivsobolev} guarantees that the fractional chain rule can be imported directly from the Euclidean setting, with an appropriate restriction on the range of Lebesgue exponents when $a$ is negative.  The exotic Strichartz estimates however are derived from the $L^1_x\to L^\infty_x$ dispersive estimate for the propagator $e^{it\Delta}$.  Dispersive estimates were proved for the propagator $e^{-it\ld}$ for $a>0$ in \cite{Fanelli}; they were shown to fail for $a<0$ in \cite{PSS1}.

\section{Linear profile decomposition}\label{S:LPD}

The goal of this section is to establish the linear profile decomposition for the propagator $e^{-it\ld}$ associated to bounded sequences in $\dot H^1_a(\R^3)$. 

\begin{theorem}[$\dot H^1_a(\R^3)$ linear profile decomposition]\label{T:LPD}
Let $\{f_n\}$ be a bounded sequence in $\dot H^1_a(\R^3)$.  After passing to a subsequence, there exist $J^*\in \{0, 1, 2,
\ldots,\infty\}$, non-zero profiles $\{\phi^j\}_{j=1}^{J^*}\subset \dot H^1(\R^3)$, $\{\lambda_n^j\}_{j=1}^{J^*}\subset(0,\infty)$, and
$\{(t_n^j,x_n^j)\}_{j=1}^{J^*}\subset \R\times \R^3$ such that for each finite $0\le J\le J^*$, we have the decomposition
\begin{align*}
f_n=\sum_{j=1}^ J \phi_n^j +w_n^J \qtq{with} \phi_n^j=G_n^j\big[e^{-it_n^j\ld^{n_j}}\phi^j\big] \qtq{and} w_n^J\in \dot H^1_a(\R^3)
\end{align*}
satisfying
\begin{align*}
\lim_{J\to J^*} \limsup_{n\to\infty}\|e^{-it\ld}w_n^J\|_{L_{t,x}^{10}(\R\times\R^3)}=0,\\
\lim_{n\to\infty}\Bigl\{\|f_n\|_{\dot H^1_a(\R^3)}^2-\sum_{j=1}^J\|\phi_n^j\|_{\dot H_a^1(\R^3)}^2 -\|w_n^J\|_{\dot H^1_a(\R^3)}^2\Bigr\}=0, \\
\lim_{n\to\infty}\Bigl\{\|f_n\|_{L^6_x(\R^3)}^6-\sum_{j=1}^J \|\phi_n^j\|_{L^6_x(\R^3)}^6-\|w_n^J\|_{L^6_x(\R^3)}^6\Bigr\}=0.
\end{align*}
Here, $\ld^{n_j}$ is as in Definition~\ref{D:Ln} with $y_n^j=\frac{x_n^j}{\lambda_n^j}$ and $[G_n^j h] (x) :=  (\lambda_n^j)^{-\frac 12} h\bigl(\tfrac{x-x_n^j}{\lambda_n^j} \bigr)$. Moreover, for all $j\neq k$  we have the asymptotic orthogonality
property
\begin{align}\label{E:LP5}
\bigl|\log\tfrac{\lambda_n^j}{\lambda_n^k} \bigr| +
    \tfrac{|x_n^j-x_n^k|^2}{\lambda_n^j\lambda_n^k}+\tfrac{|t_n^j(\lambda_n^j)^2-t_n^k(\lambda_n^k)^2|}{\lambda_n^j\lambda_n^k}\to\infty \qtq{as} n\to \infty.
\end{align}
In addition, we may assume that for each $j$ either $t_n^j\equiv 0$ or $t_n^j\to \pm \infty$.
\end{theorem}

\begin{definition}\label{D:Ln}
Given a sequence $\{y_n\}\subset\R^3$, we define
\begin{equation*}
\ld^n:=-\Delta+ \frac{a}{|x+y_n|^2} \qtq{and} \ld^\infty:=\begin{cases} -\Delta+\frac{a}{|x+y_\infty|^2} \quad &\text{if}\quad y_n\to y_\infty\in\R^3,\\
-\Delta&\text{if}\quad |y_n|\to\infty.
\end{cases}
\end{equation*}
\end{definition}

The following typifies the manner in which the operators $\ld^n$ appear: for any $y_n\in\R^3$ and $N_n>0$,
\begin{equation}\label{equ:transcale}
N_n^{\frac12}e^{-it\ld}\big[\phi(N_nx-y_n)\big]=N_n^{\frac12} \big[e^{-iN_n^{2}t\ld^n}\phi\big]\big(N_nx-y_n\big).
\end{equation}

The operators $\ld^\infty$ appear as limits of the operators $\ld^n$ in various manners that we now justify:

\begin{lemma}[Convergence of operators]\label{L:convg op}
Fix $a>-\frac14$.  Suppose we are given sequences $t_n\to t\in\R$ and $y_n\to y_\infty \in \R^3\cup\{\infty\}$.  With $\ld^n$ and $\ld^\infty$ as in Definition~\ref{D:Ln} we have\begin{align}
\lim_{n\to\infty}\big\|\ld^n \psi-\ld^\infty \psi\big\|_{\dot H^{-1}(\R^3)}&=0 \quad\text{for all $\psi\in \dot H^1(\R^3)$}, \label{convg op1}\\
\lim_{n\to\infty}\big\|(e^{it_n\ld^n}-e^{it\ld^\infty})\psi\big\|_{\dot H^{-1}(\R^3)}&=0  \quad\text{for all $\psi\in \dot H^{-1}(\R^3)$}, \label{convg op2}\\
\lim_{n\to\infty}\big\| \bigl[\sqrt{\ld^n} - \sqrt{\ld^\infty} \,\bigr] \psi\big\|_{L^2(\R^3)}&=0 \quad\text{for all $\psi\in \dot H^1(\R^3)$}, \label{convg op5}
\end{align}
and for all admissible pairs $\frac2q+\frac3r=\frac32$ with $2< q\leq\infty$ we have
\begin{gather}
\lim_{n\to\infty}\big\|(e^{it\ld^n}-e^{it\ld^\infty})\psi\big\|_{L^q_t L^r_x(\R\times\R^3)}=0 \quad\text{for all $\psi\in L^2(\R^3)$}.\label{convg op4}
\end{gather}
If additionally $y_\infty\neq 0$, then
\begin{gather}
\lim_{n\to\infty}\big\|\big[e^{-\ld^n}-e^{-\ld^\infty}\big] \delta_0\big\|_{\dot H^{-1}(\R^3)}=0. \label{convg op3}
\end{gather}
\end{lemma}

\begin{proof}
By Theorem~\ref{pro:equivsobolev}, the operators $\ld^n$ and $\ld^\infty$ are isomorphisms of $\dot H^1_x$ onto $\dot H^{-1}_x$ with bounds independent of $n$.  Consequently, it suffices to prove \eqref{convg op1} for the dense subclass $C^\infty_c(\R^3\setminus\{y_\infty\})$ of $\dot H^1(\R^3)$.  (In the case $y_\infty=\infty$, we understand $\R^3\setminus\{y_\infty\}$ as $\R^3$.)  In this subclass, convergence in $L^{6/5}_x \hookrightarrow \dot H^{-1}_x$ is trivial.

We next prove \eqref{convg op4} in the case $q=\infty$ and $r=2$.  Note that all other cases of \eqref{convg op4} then follow by interpolation with the end-point Strichartz inequality.  We also note that the Strichartz inequality permits us to assume that $\psi\in C^\infty_c(\R^3)$.

Let us suppose first that $y_n\to\infty$ so that $\ld^\infty=-\Delta$.  By the Duhamel formula and Strichartz,
\begin{equation}\label{E:duh}
\big\|(e^{it\ld^n}-e^{-it\Delta})\psi\big\|_{L^\infty_t L^2_x(\R\times\R^3)} \lesssim \big\|\tfrac{a}{|x+y_n|^2}e^{-it\Delta}\psi\big\|_{L_{t,x}^\frac{10}7(\R\times\R^3)}. 
\end{equation}
On the other hand, as $\psi\in C^\infty_c(\R^3)$, we have
$$
| [e^{-it\Delta}\psi](x)| \lesssim_\psi \langle t\rangle^{-3/2} \bigl(1+\tfrac{|x|}{\langle t\rangle}\bigr)^{-m}
$$
for any $m>0$. That RHS\eqref{E:duh} converges to zero as $n\to\infty$ follows easily from this and the fact that $y_n\to\infty$.

Let us suppose now that $y_n\to y_\infty\in\R^3$.  Translating $\psi$ if necessary, we may assume that $y_\infty=0$; this implies $\ld^\infty=\ld$.  Adopting the notation
\begin{equation*}
V_n(x) = \tfrac{a}{|x+y_n|^2}-\tfrac{a}{|x|^2} = - a \tfrac{y_n\cdot(2x+y_n)}{|x|^2|x+y_n|^2}
\end{equation*}
and proceeding as previously using the Duhamel formula and Strichartz inequality, we obtain
\begin{equation}\label{E:duh'}
\big\|(e^{it\ld^n}-e^{it\ld})\psi\big\|_{L^\infty_t L^2_x(\R\times\R^3)} \lesssim \big\|V_n e^{it\ld}\psi\big\|_{L^2_t L^{\frac65}_x(\R\times\R^3)}. 
\end{equation}
Given any $\eps>0$ and $|y_n|<\frac\eps2$, we have
\begin{align*}
\text{RHS\eqref{E:duh'}} &\lesssim  \|V_n\|_{L^{\frac{15}{11}}_x(\{|x|\leq \eps\} }  \| e^{it\ld}\psi \Big\|_{L_t^2L_x^{10}}
	+ \|V_n\|_{L^{\frac32}_x(\{|x|\geq \eps\} }  \| e^{it\ld}\psi \Big\|_{L_t^2L_x^{6}}\\
&\lesssim \eps^{\frac15} \big\| (\ld)^{1/10} \psi\big\|_{L_x^2} + |y_n| \eps^{-1} \|\psi\|_{L_x^2} \\
&\lesssim_\psi \eps^{\frac15} + o(1)\qtq{as} n\to\infty.
\end{align*}
As $\eps>0$ was arbitrary, we deduce that $\text{RHS\eqref{E:duh'}}\to 0$ as $n\to\infty$, thereby completing the proof of \eqref{convg op4}.

Next we look at \eqref{convg op3}.  Notice that $y_\infty\neq 0$ is necessary when $\sigma>0$; otherwise, $e^{-\ld^\infty}\delta_0 \equiv \infty$ as can be seen from Theorem~\ref{T:heat}.  Let us consider first the case $y_n\to\infty$, for which $\ld^\infty=-\Delta$.  From Theorem~\ref{pro:equivsobolev}, we see that $e^{-s\ld^n}$ is bounded on $\dot H^{-1}_x$ uniformly in $n$ and $s\geq 0$.   From the Duhamel formula and Sobolev embedding, we may then deduce that
\begin{align*}
\big\|\big[e^{-\ld^n}-e^{-\ld^\infty}\big] \delta_0\big\|_{\dot H^{-1}_x}
	&\lesssim \int_0^1 \big\| e^{-(1-s)\ld^n} \tfrac{a}{|x+y_n|^2} e^{s\Delta} \delta_0\big\|_{\dot H^{-1}_x}\,ds \\
&\lesssim \big\| \tfrac{a}{|x+y_n|^2} e^{s\Delta} \delta_0\big\|_{L^1_s L^{6/5}_x ([0,1]\times\R^3)}.
\end{align*}
To estimate this quantity, we first observe that for $|y_n|\geq 2R >0$, we have
\begin{align*}
\| \tfrac{a}{|x+y_n|^2} e^{s\Delta} \delta_0 \|_{L^{6/5}_x(\{|x+y_n|<R\})}
	&\lesssim  \| \tfrac{a}{|x+y_n|^2} \|_{L^{6/5}_x(\{|x+y_n|<R\})} \| e^{s\Delta} \delta_0 \|_{L^\infty_x(\{|x+y_n|<R\})} \\
	&\lesssim R^{1/2} s^{-3/2} e^{-c\frac{|y_n|^2}{s}},
\end{align*} 
which converges to zero in $L^1_s([0,1])$ as $n\to \infty$, irrespective of $R>0$.  On the other hand,
\begin{align*}
\| \tfrac{a}{|x+y_n|^2} e^{s\Delta} \delta_0 \|_{L^{6/5}_x(\{|x+y_n|>R\})}
	&\lesssim  \| \tfrac{a}{|x+y_n|^2} \|_{L^{3}_x(\{|x+y_n|>R\})} \| e^{s\Delta} \delta_0 \|_{L^2_x(\R^3)} \lesssim R^{-1} s^{-\frac34},
\end{align*}
which may be made arbitrarily small in $L^1_s([0,1])$ by choosing $R$ large.  This completes the proof of \eqref{convg op3} in the case $y_n\to\infty$.

Suppose now that $y_n\to y_\infty\in \R^3\setminus \{0\}$.  Proceeding as in the previous case, the problem reduces to showing that
$$
\big\| V_n e^{-s\ld^\infty} \delta_0\big\|_{L^1_s L^{6/5}_x ([0,1]\times\R^3)} \to 0 \qtq{as} n\to\infty,
$$
where
\begin{equation*}
V_n(x) = \tfrac{a}{|x+y_n|^2}-\tfrac{a}{|x+y_\infty|^2} = a \tfrac{(y_\infty-y_n)\cdot(2x+y_n+y_\infty)}{|x+y_\infty|^2|x+y_n|^2}.
\end{equation*}
This can be shown in a manner similar to that used in the previous case.  Suppose $0<\eps<\frac13|y_\infty|$ and $|y_n-y_\infty|<\frac13\eps$, then
\begin{align*}
\| V_n e^{-s\ld^\infty} \delta_0 \|_{L^{6/5}_x(\{|x+y_\infty|>\eps\})}
	&\lesssim  \| V_n \|_{L^\infty_x(\{|x+y_\infty|>\eps\})} \| e^{-s\ld^\infty} \delta_0 \|_{L^{6/5}_x(\R^3)} \\
	&\lesssim |y_n-y_\infty| \eps^{-3} s^{-1/4},
\end{align*}
which converges to zero in $L^1_s([0,1])$ as $n\to \infty$, irrespective of $\eps>0$, while
\begin{align*}
\| V_n e^{-s\ld^\infty} \delta_0 \|_{L^{6/5}_x(\{|x+y_\infty|<\eps\})}
	&\lesssim  \| V_n \|_{L^{6/5}_x(\{|x+y_\infty|<\eps\})} \| e^{-s\ld^\infty} \delta_0 \|_{L^\infty_x(\{|x+y_\infty|<\eps\})} \\
	&\lesssim \eps^{1/2} s^{-3/2} \bigl(1+\tfrac{\sqrt s}{\eps}\bigr)^{2(\sigma\vee 0)} e^{-\frac c{s}|y_\infty|^2},
\end{align*}
which converges to zero in $L^1_s([0,1])$ as $\eps\to0$.  This completes the proof of \eqref{convg op3}.

We turn now to the proof of \eqref{convg op5}; once again it suffices to consider $\psi\in C^\infty_c(\R^3\setminus\{y_\infty\})$.  Given $\eps>0$, let us define
$$
f_\eps(E) := \begin{cases} \sqrt{|E|} &\text{ if } |E|<\eps^{-2} \\ \eps^{-1} &\text{ if } |E|\geq \eps^{-2}.\end{cases}
$$
By mollification and then Fourier inversion, there is a finite (signed) measure $d\mu_\eps$ (with $\mu_\eps(\{0\})=\eps^{-1}$) so that
$$
\sup_{E\in\R} \biggl| f_\eps(E) - \int e^{itE} \,d\mu_\eps(t)\biggr| \leq \eps.
$$
Combining this with \eqref{convg op4} yields
$$
\limsup_{n\to\infty} \bigl\| f_\eps(\ld^\infty)\psi - f_\eps(\ld^n)\psi \bigr\|_{L^2_x} \leq 2\eps \|\psi\|_{L^2_x}.
$$
On the other hand, using the fact that $|f_\eps(E) - \sqrt{|E|}|\leq \eps |E|$ we get
\begin{align*}
\bigl\| \sqrt{\ld^n}\psi - f_\eps(\ld^n)\psi \bigr\|_{L^2_x} &+ \bigl\| \sqrt{\ld^\infty}\psi - f_\eps(\ld^\infty)\psi \bigr\|_{L^2_x}\\
&\leq  \eps \bigl\| \ld^\infty \psi \bigr\|_{L^2_x} + \eps \bigl\| \ld^n \psi \bigr\|_{L^2_x} \to 2\eps \bigl\| \ld^\infty \psi \bigr\|_{L^2_x}
\end{align*}
as $n\to\infty$.  As $\eps>0$ was arbitrary, this completes the proof of \eqref{convg op5}.

Note that we also have the following analogue of \eqref{convg op5}:
\begin{equation}
\lim_{n\to\infty}\big\| \bigl[\sqrt{\ld^n} - \sqrt{\ld^\infty} \,\bigr] \phi\big\|_{\dot H^{-1}(\R^3)}=0 \quad\text{for all $\phi\in L^2(\R^3)$} \label{convg op6}.
\end{equation}
Indeed, any such $\phi\in L^2_x$ can be written as $\sqrt{\ld^\infty} \psi$ for some $\psi\in \dot H^1_x$.  The claim then follows by expanding
$$
\bigl[ \sqrt{\ld^n} - \sqrt{\ld^\infty} \,\bigr] \phi = \sqrt{\ld^n} \bigl[\sqrt{\ld^\infty} - \sqrt{\ld^n} \,\bigr] \psi + \bigl[\ld^n - \ld^\infty\bigr]\psi
$$
and then applying \eqref{convg op5} and \eqref{convg op1}.

Now only one part of the lemma still requires justification, namely, \eqref{convg op2}.  Theorem~\ref{pro:equivsobolev} guarantees that any $\psi\in \dot H^{-1}_x$ can be written as $\sqrt{\ld^\infty} \phi$ for some $\phi\in L^2_x$.  Decomposing
\begin{align*}
(e^{it_n\ld^n}-e^{it\ld^\infty}) \sqrt{\ld^\infty} \phi &= e^{it_n\ld^n} \bigl[ \sqrt{\ld^\infty} - \sqrt{\ld^n} \bigr]\phi
	+ \sqrt{\ld^n} \bigl[e^{it_n\ld^n} - e^{it_n\ld^\infty}\bigr]\phi\\
&\quad{} + \sqrt{\ld^n} \bigl[e^{it_n\ld^\infty} - e^{it\ld^\infty}\bigr]\phi + \bigl[\sqrt{\ld^n} - \sqrt{\ld^\infty}\bigr] e^{it \ld^\infty}\phi
\end{align*}
and applying \eqref{convg op6}, \eqref{convg op4}, and Theorem~\ref{pro:equivsobolev}, we deduce that
$$
\limsup_{n\to\infty}\big\|(e^{it_n\ld^n}-e^{it\ld^\infty})\psi\big\|_{\dot H^{-1}_x}
	\leq \limsup_{n\to\infty}\big\|(e^{it_n\ld^\infty}-e^{it\ld^\infty})\phi\big\|_{L^2_x}.
$$
This last limit is easily seen to vanish by applying the spectral theorem.  This completes the proof of \eqref{convg op2} and with it, that of Lemma~\ref{L:convg op}.
\end{proof}

Our next result is essential for proving the decoupling of the potential energy in Theorem~\ref{T:LPD}.

\begin{corollary}\label{C:L6}
Fix $a>-\frac14$.  Given $\psi\in \dot H^1_x$, $t_n\to \pm\infty$, and any sequence $\{y_n\}\subset\R^3$, we have
$$
\| e^{it_n\ld^n} \psi \|_{L^6_x} \to 0 \qtq{as} n\to\infty,
$$
where $\ld^n$ is as in Definition~\ref{D:Ln}.
\end{corollary}

\begin{proof}
Without loss of generality, we may assume that $y_n\to y_\infty\in \R^3\cup \{\infty\}$.  Let $\ld^\infty$ be as in Definition~\ref{D:Ln}.

Writing
\begin{align*}
\sqrt{\ld^\infty} \bigl[e^{it_n\ld^n} \psi - e^{it_n\ld^\infty} \psi\bigr]  &=  [\sqrt{\ld^\infty}- \sqrt{\ld^n}]e^{it_n\ld^n} \psi +e^{it_n\ld^n} [\sqrt{\ld^n}- \sqrt{\ld^\infty}]\psi\\
&\quad+ [e^{it_n\ld^n} - e^{it_n\ld^\infty}]\sqrt{\ld^\infty}\psi
\end{align*}
and using \eqref{convg op5}, and \eqref{convg op4}, we see that $e^{it_n\ld^n} \psi - e^{it_n\ld^\infty} \psi\to 0$ in $\dot H^1_x$.  Thus, by Sobolev embedding it suffices to show
\begin{align}\label{941}
\| e^{it_n\ld^\infty} \psi \|_{L^6_x} \to 0 \qtq{as} n\to\infty;
\end{align}
moreover, by density, we may assume $\psi\in C_c^\infty(\R^3\setminus\{y_\infty\})$.

Let $F(t):=\| e^{it\ld^\infty} \psi \|_{L^6_x}$.  By the Strichartz inequality, $F\in L_t^2(\R)$.  Moreover, $F$ is Lipschitz; indeed, by Sobolev embedding, 
\begin{align*}
|\partial_t F(t)| \leq \| \partial_t[e^{it\ld^\infty} \psi] \|_{L^6_x} \lesssim \| \ld^\infty \psi\|_{\dot H^1_x}\lesssim_\psi 1.
\end{align*}
Therefore, $F(t_n)\to 0$ as $n\to \infty$.  This completes the proof of the corollary.
\end{proof}

As we are dealing with an energy-critical problem, we also need convergence of propagators in suitable energy-critical spaces.  For this we need the assumption $a>-\frac14+\frac1{25}$.

\begin{corollary}\label{C:L10}
Fix $a>-\frac14+\frac1{25}$.  Suppose we are given $y_n\to y_\infty \in \R^3\cup\{\infty\}$ and let $\ld^n$ and $\ld^\infty$ be as in Definition~\ref{D:Ln}.  Then
\begin{align*}
\lim_{n\to\infty}\big\|(e^{it\ld^n}-e^{it\ld^\infty})\psi\big\|_{L_{t,x}^{10}(\R\times\R^3)}=0  \quad\text{for all $\psi\in \dot H^1(\R^3)$}.
\end{align*}
\end{corollary}

\begin{proof}
By Strichartz, it suffices to prove the claim for $\psi\in C_c^\infty(\R^3\setminus\{y_\infty\})$.  By H\"older,
\begin{align*}
\big\|(e^{it\ld^n}&-e^{it\ld^\infty})\psi\big\|_{L_{t,x}^{10}} \leq \big\|(e^{it\ld^n}-e^{it\ld^\infty})\psi\big\|_{L_{t,x}^{\frac{10}3}}^\theta\big\|(e^{it\ld^n}-e^{it\ld^\infty})\psi\big\|_{L_{t,x}^{q}}^{1-\theta}
\end{align*}
for any $q>10$ and $\theta\in(0,1)$ given by $\frac1{10} = \frac{3\theta}{10} + \frac{1-\theta}q$.  By \eqref{convg op4}, the first term on the right hand side above converges to zero as $n\to \infty$.  On the other hand, by Sobolev embedding, Theorem~\ref{pro:equivsobolev}, and Strichartz,
\begin{align*}
\big\|(e^{it\ld^n}-e^{it\ld^\infty})\psi\big\|_{L_{t,x}^{q}(\R\times\R^3)} \lesssim \big\| |\nabla|^{\frac32-\frac5q} \psi\|_{L^2_x} \lesssim 1,
\end{align*}
provided $q$ is chosen sufficiently close to $10$ so that Theorem~\ref{pro:equivsobolev} may be applied.
\end{proof}

To prove Theorem~\ref{T:LPD}, we follow the general method outlined in \cite{KVnote}.  For the problems treated in those notes, all defects of compactness in the Strichartz inequality are associated with symmetries of the equation.  The presence of a potential in \eqref{equ1.1} breaks the translation symmetry.  The new difficulties that this introduces are reminiscent of those overcome in \cite{KVZ12}, which derived a linear profile decomposition in the exterior of a convex obstacle.  This earlier work provides inspiration (but little technical overlap) for much of what follows.

The bulk of the proof of Theorem~\ref{T:LPD} is effected by the inductive application of an inverse Strichartz inequality; see Proposition~\ref{P:inverse Strichartz} below.  This in turn relies of a refinement of the classical Strichartz inequality:

\begin{lemma}[Refined Strichartz]\label{lm:refs} Fix $a> -\frac14+\frac1{25}$. For $f\in \dot H^1_a(\R^3)$, we have
\begin{align*}
\|e^{-it\ld} f\|_{L^{10}_{t,x}(\R\times\R^3)}\lesssim \|f\|_{\dot H^1_a(\R^3)}^{\frac 15} \sup_{N\in 2^{\Z}}\|e^{-it\ld} f_N\|_{L^{10}_{t,x}(\R\times\R^3)}^{\frac45}.
\end{align*}
\end{lemma}

\begin{proof}
As $a> -\frac14+\frac1{25}$, we have $\sigma<\frac3{10}$ and so we may choose $0<\eps\leq 1$ so that $\sigma < \frac{3}{10}(1-\eps)$.  Using the square function estimate Lemma~\ref{T:sq}, Lemma~\ref{L:Bernie}, and Strichartz, we obtain
\begin{align*}
\|e^{-it\ld} f\|_{L^{10}_{t,x}}^{10} \lesssim& \iint_{\R\times\R^3} \Bigl(\sum_{N\in 2^{\Z}}|e^{-it\ld} f_N|^2 \Bigr)^5 \,dx \,dt\\
\lesssim& \sum_{N_1\le \cdots\le N_5}\iint_{\R\times\R^3} |e^{-it\ld} f_{N_1}|^2 \cdots |e^{-it\ld} f_{N_5}|^2 \,dx\,dt\\
\lesssim& \sum_{N_1\le\cdots\le N_5} \|e^{-it\ld} f_{N_1}\|_{L^{10}_t L^{\frac{10}{1-\eps}}_x}
	\|e^{-it\ld} f_{N_1}\|_{L^{10}_{t,x}} \prod_{j=2}^4\|e^{-it\ld} f_{N_j}\|_{L^{10}_{t,x}}^2 \\
&\qquad\qquad \qquad\cdot\|e^{-it\ld} f_{N_5}\|_{L^{10}_{t,x}} \|e^{-it\ld} f_{N_5}\|_{L^{10}_t L^{\frac{10}{1+\eps}}_x}\\
\lesssim& \sup_{N\in 2^{\Z}}\|e^{-it\ld} f_N\|_{L^{10}_{t,x}}^8 \sum_{N_1\le N_5} \bigr[1+\log\bigl(\tfrac {N_5}{N_1}\bigr) \bigr]^3
	N_1^{1+\frac{3\eps}{10}}\| e^{-it\ld} f_{N_1}\|_{L^{10}_t L^{\frac{30}{13}}_x}\\
&\qquad\qquad\qquad\cdot N_5^{1-\frac{3\eps}{10}} \|e^{-it\ld} f_{N_5}\|_{L^{10}_t L^{\frac{30}{13}}_x} \\
\lesssim& \sup_{N\in 2^{\Z}}\|e^{-it\ld} f_N\|_{L^{10}_{t,x}}^8
\sum_{N_1\le N_5} \bigr[1+\log\bigl(\tfrac {N_5}{N_1}\bigr)\bigr]^3
\bigl(\tfrac{N_1}{N_5}\bigr)^{\frac{3\eps}{10}}
    \|f_{N_1}\|_{\dot H^1_x} \|f_{N_5}\|_{\dot H^1_x}\\
	\lesssim& \sup_{N\in 2^{\Z}}\|e^{-it\ld} f_N\|_{L^{10}_{t,x}}^8\|f\|_{\dot H^1_x}^2,
\end{align*}
where all space-time norms are over $\R\times\R^3$.  
\end{proof}

\begin{proposition}[Inverse Strichartz inequality]\label{P:inverse Strichartz}
Assume $a> -\frac14+\frac1{25}$. Let $\{f_n\}\subset \dot H^1_a(\R^3)$ be such that
\begin{align*}
\lim_{n\to \infty}\|f_n\|_{\dot H^1_a(\R^3)}=A < \infty \quad{and}~~
\lim_{n\to\infty}\|e^{-it\ld} f_n\|_{L^{10}_{t,x}}=\eps >0.
\end{align*}
Then there exist a subsequence in $n$, $\phi\in \dot H^1(\R^3)$, $\{N_n\}\subset 2^{\Z}$, and $\{(t_n, x_n)\}\subset
\R\times\R^3$  such that
\begin{align}\label{linearweak}
&g_n(x):=N_n^{-\frac12}\big(e^{-it_n\ld}f_n\big)\big(x_n+\tfrac{x}{N_n}\big)\rightharpoonup~\phi(x) \quad \text{weakly in}~~\dot H^1(\R^3),\\\label{dech}
&\lim_{n\to\infty}\Bigl\{ \|f_n\|_{\dot H^1_a(\R^3)}^2-\|f_n-\phi_n\|_{\dot H^1_a(\R^3)}^2\Bigr\} \gtrsim_{\eps,A} 1,\\
&\lim_{n\to\infty}\Bigl\{ \|f_n\|_{L^6_x}^6-\|f_n-\phi_n\|_{L^6_x}^6-\|\phi_n\|_{L^6_x}^6\Bigr\} =0,\label{dect}
\end{align}
where
\begin{equation*}
\phi_n(x)=N_n^\frac12e^{it_n\ld}\big[\phi\big(N_n(x-x_n)\big)\big]=N_n^\frac12\big(e^{iN_n^2t_n\ld^n}\phi\big)\big(N_n(x-x_n)\big),
\end{equation*}
with $\ld^n$ as in Definition~\ref{D:Ln} with $y_n=N_nx_n$.  Moreover, we may assume that either $N_n^2t_n\to \pm \infty$ or $t_n\equiv 0$ and that either $N_n|x_n|\to \infty$ or $x_n\equiv 0$.
\end{proposition}

\begin{proof}
By Lemma \ref{lm:refs}, for $n$ sufficiently large there exists $N_n\in 2^{\Z}$ such that
\begin{align}\label{equ:jiz}
\|e^{-it\ld} P_{N_n}^a f_n\|_{L^{10}_{t,x}}\gtrsim \eps^{\frac54}A^{-\frac 14}.
\end{align}
By the Mikhlin multiplier theorem for $\ld$, there is an $L_x^{10}$-bounded multiplier $m$ so that $P_N^a=m(\ld)\tilde P_N^a$.  Thus, \eqref{equ:jiz} implies
\begin{align}\label{equ:jiz'}
\|e^{-it\ld} \tilde P_{N_n}^a f_n\|_{L^{10}_{t,x}}\gtrsim \eps^{\frac54}A^{-\frac 14}.
\end{align}

Using the heat kernel bounds provided by Theorem~\ref{T:heat}, it is easy to check that
$$
\|\tilde P_{N_n}^a f\|_{L_x^{10}(\{|x|\leq\alpha N_n^{-1}\})}\lesssim \alpha^{\frac3{10}-\sigma}\|f\|_{L_x^{10}}.
$$
Therefore, \eqref{equ:jiz'} guarantees the existence of an $f_n$-independent $\eta>0$ so that 
\begin{align}\label{equ:jiz''}
\|e^{-it\ld} \tilde P_{N_n}^a f_n\|_{L^{10}_{t,x}(\R\times\{|x|\geq\eta (\eps/A)^c N_n^{-1}\})}\gtrsim \eps^{\frac54}A^{-\frac 14},
\end{align}
where $c=\frac5{2(3-10\sigma)}$.

On the other hand, by the Strichartz and Bernstein inequalities
\begin{align*}
\|e^{-it\ld} \tilde P_{N_n}^a f_n\|_{L_{t,x}^{\frac{10}3}(\R\times\R^3)}\lesssim \| \tilde P_{N_n}^a f_n\|_{L_x^2(\R^3)} \lesssim N_n^{-1} A.
\end{align*}
This together with \eqref{equ:jiz''} and H\"older's inequality yields
\begin{align*}
A^{-\frac 14}\eps^{\frac 54}&\lesssim \|e^{-it\ld} \tilde P_{N_n}^a f_n\|_{L^{10}_{t,x}(\R\times\{|x|\geq\eta (\eps/A)^c N_n^{-1}\})}\\
&\lesssim  \|e^{-it\ld} \tilde P_{N_n}^af_n\|_{L_{t,x}^{\frac{10}3}(\R\times\R^3)}^{\frac 13}
	\|e^{-it\ld} \tilde P_{N_n}^a f_n\|_{L_{t,x}^{\infty} (\R\times\{|x|\geq\eta (\eps/A)^c N_n^{-1}\})}^{\frac 23} \\
&\lesssim  N_n^{-\frac 13}A^{\frac 13}\|e^{-it\ld} \tilde P_{N_n}^a f_n\|_{L_{t,x}^{\infty}(\R\times\{|x|\geq\eta (\eps/A)^c N_n^{-1}\})}^{\frac 23}.
\end{align*}
Therefore, there exist $\tau_n\in\R$ and $x_n\in \R^3$ with $N_n |x_n|\gtrsim (\eps/A)^c$ such that
\begin{align}\label{cncen}
N_n^{-\frac12}\Bigl|( \tilde P_{N_n}^a e^{-i\tau_n\ld} f_n)(x_n)\Bigr|\gtrsim \eps (\tfrac{\eps}A)^{\frac 78}.
\end{align}
Passing to a subsequence, we may assume $N_n^2\tau_n\to \tau_\infty\in[-\infty,\infty]$.  When $\tau_\infty$ is finite, we set $t_n:= 0$; otherwise, we define $t_n:=\tau_n$.

Having chosen the parameters $t_n$, $x_n$, and $N_n$, we now proceed to the construction of the profile $\phi$.  The inequality \eqref{cncen} will underlie the demonstration that $\phi$ carries nontrivial energy.

With $g_n$ defined as in \eqref{linearweak}, 
\begin{align*}
\|g_n\|_{\dot H^1_x}=\|\nabla e^{-it_n\ld}f_n\|_{L^2_x}\sim \|\sqrt{\ld} e^{-it_n\ld}f_n\|_{L^2_x} =\|f_n\|_{\dot H^1_a}\lesssim A,
\end{align*}
and so, passing to a subsequence, there exists $\phi$ so that $g_n\rightharpoonup \phi$ weakly in $\dot H^1_x$.  This proves \eqref{linearweak}.

Next, we turn to proving \eqref{dech}.  Changing variables, then using Lemma~\ref{L:convg op} and the fact that $g_n\rightharpoonup \phi$ weakly in $\dot H^1_x$, we get
\begin{align*}
\|f_n\|_{\dot H^1_a}^2-\|f_n-\phi_n\|_{\dot H^1_a}^2
=2\Re \langle g_n, \ld^n \phi \rangle - \langle \phi,  \ld^n \phi \rangle \to \langle \phi,  \ld^\infty \phi \rangle \quad \text{as } n\to \infty,
\end{align*}
where $\ld^\infty$ is as in Definition~\ref{D:Ln}. Thus, it suffices to prove a lower bound on the $\dot H^1_x$-norm of $\phi$.

Setting
$$
h_n(x) = \begin{cases}
N_n^{-3}\bigl[ e^{-i \tau_n\ld} \tilde P_{N_n}^a\delta_{x_n}\bigr]  (x_n+ \tfrac{x}{N_n})= \bigl[ e^{-i N_n^2\tau_n\ld^n} \tilde P_1^n \delta_{0}\bigr]  (x) & \text{if }\tau_\infty\in\R \\
N_n^{-3}\bigl[ \tilde P_{N_n}^a \delta_{x_n}\bigr]  (x_n+ \tfrac{x}{N_n})= \bigl[\tilde P_1^n \delta_{0}\bigr]  (x) & \text{if }\tau_\infty=\pm\infty
\end{cases}
$$
with $\tilde P_1^n:= e^{-\ld^n} - e^{-4\ld^n}$, and performing a change of variables, \eqref{cncen} becomes
$$
\big|\langle h_n, g_n\rangle\big| \gtrsim \eps (\tfrac{\eps}A)^{\frac 78}.
$$
By construction, $g_n\rightharpoonup\phi$ in $\dot H^1_x$.   We set $\tilde P_1^\infty:= e^{-\ld^\infty} - e^{-4\ld^\infty}$.  By Lemma~\ref{L:convg op}, $h_n \to \tilde P_1^\infty \delta_0$ in $\dot H^{-1}_x$ when $\tau_\infty=\pm\infty$ and $h_n \to e^{-i\tau_\infty\ld^\infty} \tilde P_1^\infty \delta_0$ in $\dot H^{-1}_x$ when $\tau_\infty$ is finite.  Thus when $\tau_\infty$ is finite, we deduce that
$$
\eps (\tfrac{\eps}A)^{\frac 78} \lesssim \big|\langle \tilde P_1^\infty \delta_0, e^{+i\tau_\infty\ld^\infty} \phi\rangle\big| \lesssim \| \phi\|_{\dot H^1_x} \| \tilde P_1^\infty\delta_0 \|_{L^{6/5}_x}.
$$
In view of Theorem~\ref{T:heat} and the fact that $N_n|x_n| \gtrsim (\eps/A)^c$, we have
$$
\|\tilde P_1^\infty\delta_0\|_{L^{6/5}_x}\sim \bigl[1+\bigl(\tfrac\eps{A}\bigr)^{-c}\bigr]^\sigma,
$$
which completes the proof of \eqref{dech} when $\tau_\infty$ is finite.  Trivial modifications of the last steps handle the case $\tau_\infty=\pm\infty$. 

We now turn to \eqref{dect}.  If $t_n\equiv 0$, then by Rellich--Kondrashov (which yields $g_n\to \phi$ a.e.) and Lemma~\ref{reflemms} we have
$$
\|g_n\|_{L^6_x}^6-\|g_n-\phi\|_{L^6_x}^6-\|\phi\|_{L^6_x}^6 \to 0 \qtq{as} n\to \infty.
$$
A change of variables then yields \eqref{dect} in this case.  If $t_n=\tau_n$, then it suffices to observe that $\phi_n\to 0$ in $L^6_x$ by Corollary~\ref{C:L6}.

Finally, passing to another subsequence if necessary, we may assume that either $N_n |x_n|\to \infty$ or $N_nx_n\to y_\infty\in \R^3$.  In the latter case, we may take $x_n\equiv 0$ by replacing the profile $\phi$ found previously by $\phi(x-y_\infty)$.
\end{proof}

With Proposition~\ref{P:inverse Strichartz} in place, the greater part of the proof of Theorem~\ref{T:LPD} amounts to careful book-keeping.  We omit the details; they can be found, for example, in \cite{KVZ12,Oberbook}.  However, as one sees there, two extra facts about the free propagator are needed to verify \eqref{E:LP5}.  For the problem at hand, these two additional inputs are provided by Lemmas~\ref{L:compact} and~\ref{L:converg} below; they are the analogues of Lemmas~5.5 and~5.4 in \cite{KVZ12} and play the same roles as these earlier results in the proof of \eqref{E:LP5}.

\begin{lemma}[Weak convergence]\label{L:compact}
Let $f_n\in \dot H^1(\R^3)$ be such that $f_n\rightharpoonup 0$ weakly in $\dot H^1(\R^3)$ and let $t_n\to t_\infty\in \R$. Then for any $y_n\in \R^3$,
\begin{align*}
e^{-it_n\ld^n} f_n\rightharpoonup 0 \quad\text{weakly in}~~ \dot H^1(\R^3),
\end{align*}
where $\ld^n$ is as in Definition~\ref{D:Ln}.
\end{lemma}

\begin{proof}
Without loss of generality, we may assume that $y_n\to y_\infty\in \R^3\cup \{\infty\}$.  Let $\ld^\infty$ be as in Definition~\ref{D:Ln}.
For any $\psi\in C_c^{\infty}(\R^3\setminus\{y_\infty\})$, we have
\begin{align*}
\bigl|\bigl\langle [e^{-it_n\ld^n}-e^{-it_\infty \ld^n}]f_n, \psi\bigr\rangle_{\dot H^1_x}\bigr|
&\lesssim \big\|[e^{-it_n\ld^n}-e^{-it_\infty \ld^n}]f_n\big\|_{L^2_x} \|\Delta\psi\|_{L^2_x}\\
&\lesssim |t_n-t_\infty|^{1/2} \|\sqrt{\ld^n}f_n\|_{L^2_x}\|\Delta\psi\|_{L^2_x},
\end{align*}
which converges to zero as $n\to \infty$.  To see the last inequality above, we used the spectral theorem together with the elementary inequality
$$
|e^{-it_n\lambda}-e^{-it_\infty\lambda}|\lesssim |t_n-t_\infty|^{1/2}\lambda^{1/2} \qtq{for}\lambda\geq 0.
$$

Thus, to prove the lemma it suffices to prove that $e^{-it_\infty \ld^n} f_n\rightharpoonup 0$ in $\dot H^1_x$. For $\psi\in C_c^\infty(\R^3\setminus\{y_\infty\})$,
\begin{align*}
\langle e^{-it_\infty \ld^n} f_n, \psi \rangle_{\dot H^1_x}
=\big\langle f_n, [e^{it_\infty \ld^n} -e^{it_\infty \ld^\infty}](-\Delta\psi) \big\rangle_{L^2_x} + \big\langle f_n,e^{it_\infty \ld^\infty}(-\Delta\psi) \big\rangle_{L^2_x}.
\end{align*} 
The first inner product on the right-hand side above converges to zero by Lemma~\ref{L:convg op}.  As $f_n\rightharpoonup 0$ in $\dot H^1(\R^3)$ by assumption and $e^{it_\infty \ld^\infty}\Delta\psi \in \dot H^{-1}(\R^3)$, the second inner product also converges to zero.  This completes the proof of the lemma.
\end{proof}

\begin{lemma}[Weak convergence]\label{L:converg}
Let $f\in \dot H^1(\R^3)$ and let $\{(t_n,x_n)\}_{n\geq 1}\subset\R\times\R^3$ and $\{y_n\}\subset \R^3$.  Then
\begin{align}\label{lc}
\big[e^{-it_n \ld^n}f\big](x+x_n) \rightharpoonup 0 \quad \text{weakly in $\dot H^1(\R^3)$ as $n\to \infty$,}
\end{align}
whenever $|t_n|\to \infty$ or $|x_n|\to \infty$.  Here $\ld^n$ is as in Definition~\ref{D:Ln}.
\end{lemma}

\begin{proof} Naturally, it suffices to prove \eqref{lc} along subsequences where $y_n\to y_\infty\in \R^3\cup\{\infty\}$.  Let $\ld^\infty$ be as in Definition~\ref{D:Ln}.

We first prove \eqref{lc} when $t_n\to \infty$; the proof when $t_n\to-\infty$ follows symmetrically.  Let $\psi\in C_c^\infty(\R^3\setminus\{y_\infty\})$ and let
$$
F_n(t):=\langle e^{-it \ld^n}f(x+x_n), \psi\rangle_{\dot H^1(\R^3)}.
$$
To establish \eqref{lc}, we need to show
\begin{align}\label{lc1}
F_n(t_n)\to 0 \quad\text{as}~~ n\to \infty.
\end{align}

A computation yields
\begin{align*}
|\partial_t F_n(t)|&= \bigl|\langle \ld^n e^{-it\ld^n}f(x+x_n), \Delta\psi\rangle_{L^2(\R^3)}\bigr| \lesssim \|f\|_{\dot H^1_x} \|\Delta\psi\|_{\dot H^1_x}\lesssim_{f,\psi}1.
\end{align*}
On the other hand,
\begin{align*}
\|F\|_{L_t^{\frac{10} 3}([t_n,\infty))}
&\lesssim \|e^{-it \ld^n}f\|_{L_{t,x}^{\frac{10}3}([t_n,\infty)\times\R^3)}\|\Delta \psi\|_{L_x^{\frac{10}7}(\R^3)}\\
&\lesssim_\psi \bigl\|[e^{-it\ld^n}- e^{-it\ld^\infty}]f\bigr\|_{L_{t,x}^{\frac{10}3}([t_n,\infty)\times\R^3)} +\|e^{-it\ld^\infty}f\bigr\|_{L_{t,x}^{\frac{10}3}([t_n,\infty)\times\R^3)}.
\end{align*}
The first term on the right-hand side above converges to zero by Lemma~\ref{L:convg op}.  The second term converges to zero by the Strichartz inequality combined with the monotone convergence theorem.  Putting everything together, we derive
\eqref{lc1} and so \eqref{lc} when $t_n\to \infty$.

Now assume $\{t_n\}_{n\geq 1}$ is bounded, but $|x_n|\to \infty$ as $n\to \infty$.  Without loss of generality, we may assume $t_n\to t_\infty\in \R$ as $n\to \infty$.  As
$$
\big[e^{-it_n \ld^n}f\big](x+x_n) = \Big[ e^{-it_n(-\Delta + \tfrac{a}{|\cdot+y_n+x_n|^{2}})}f(\cdot+x_n)\Big](x),
$$ 
the claim now follows from Lemma~\ref{L:compact}.
\end{proof}




\section{Embedding nonlinear profiles}\label{S:4}

In this section, we prove an analogue of Theorem~\ref{theorem} for initial data that lives far from the origin relative to its intrinsic length scale.  In this setting, one may imagine that the potential plays little role.  Indeed, we will use solutions to \eqref{nls}, whose existence is guaranteed by Theorem~\ref{T:gopher}, to approximate the solutions to \eqref{equ1.1} in this case.

To keep formulas within margins, in this section we will adopt the following abbreviation:
$$
\dot X^1(I):=L_t^{10}\dot H^{1,\frac{30}{13}}_x(I\times\R^3).
$$

\begin{theorem}[Embedding nonlinear profiles]\label{T:embed}
Fix $a>-\frac14+\frac1{25}$.  Let $\{\lambda_n\}\subset 2^{\mathbb Z}$ and $\{x_n\}\subset \R^3$ be such that $|x_n|/\lambda_n\to \infty$. Let $\{t_n\}\subset\R$ be
such that either $t_n\equiv0$ or $t_n\to \pm\infty$.  Let $\phi\in \dot H^1(\R^3)$ and define
\begin{align*}
\phi_n(x):=\lambda_n^{-\frac12}\big[e^{-it_n\ld}\phi\big]\bigl(\tfrac{x-x_n}{\lambda_n}\bigr).
\end{align*}
Then for $n$ sufficiently large there exists a global solution $v_n$ to \eqref{equ1.1}  with initial data $v_n(0)=\phi_n$ which satisfies
\begin{align*}
\|v_n\|_{L_{t,x}^{10}(\R\times\R^3)}\lesssim 1,
\end{align*}
with the implicit constant depending only on $\|\phi\|_{\dot H^1(\R^3)}$.  Furthermore, for every $\eps>0$ there exists $N_\eps\in \N$ and
$\psi_\eps\in C_c^{\infty}(\R\times\R^3)$ such that for all $n\geq N_\eps$ we have
\begin{align}\label{dense2}
\|v_n(t-\lambda_n^2 t_n,x+x_n)-\lambda_n^{-\frac12}\psi_\eps(\lambda_n^{-2}t,\lambda_n^{-1} x)\|_{\dot X^1(\R)}<\eps.
\end{align}
\end{theorem}

\begin{proof} The proof of this theorem follows the general outline of the proof of Theorem~6.3 in \cite{KVZ12}.  It contains five steps.  In the first step we select appropriate global solutions to the energy-critical NLS without potential.  In the second step we construct a putative approximate solution to \eqref{equ1.1}.  In the third step we prove that this asymptotically matches the initial data $\phi_n$ and in the fourth step we prove that it is indeed an approximate solution to \eqref{equ1.1}.  In the fifth and last step we use Theorem~\ref{thm:stability} to find $v_n$ and prove the approximation result \eqref{dense2}.

\textbf{Step 1:}  Selection of solutions to \eqref{nls}.

Let $\theta:=\frac 1{100}$.  The construction of solutions to \eqref{nls} depends on the behaviour of $t_n$.  If $t_n\equiv0$, we let $w_n$ and $w_\infty$ be the solutions to \eqref{nls} with initial data $w_n(0)=\phi_{\le (|x_n|/\lambda_n)^{\theta}}$ and $w_\infty(0)=\phi$.  If $t_n\to \pm\infty$, we let $w_n$ and $w_\infty$ be the solutions to \eqref{nls} satisfying
\begin{align*}
\|w_n(t)-e^{it\Delta}\phi_{\le (|x_n|/\lambda_n)^{\theta}}\|_{\dot H^1(\R^3)}\to0 \qtq{and} \|w_\infty(t)-e^{it\Delta}\phi\|_{\dot H^1(\R^3)}\to 0
\end{align*}
as $t\to \pm \infty$.

In all cases, Theorem~\ref{T:gopher} together with perturbation theory and persistence of regularity for \eqref{nls} implies that $w_n$ and $w_\infty$ are global solutions obeying
\begin{equation}\label{cond3}
\left\{ \quad \begin{aligned}
&\|w_n\|_{\dot S^1(\R)}+\|w_\infty\|_{\dot S^1(\R)}\lesssim 1,\\
&\lim_{n\to \infty}\|w_n-w_\infty\|_{\dot S^1(\R)}= 0,\\
&\||\nabla|^s w_n\|_{\dot S^1(\R)}\lesssim \bigl(\tfrac{|x_n|}{\lambda_n}\bigr)^{s\theta} \qtq{for all} s\ge 0.
\end{aligned} \right.
\end{equation}

\textbf{Step 2:} Constructing the approximate solution to \eqref{equ1.1}.

Fix $T>0$ to be chosen later.  We define
\begin{align*}
\tilde v_n(t,x):=\begin{cases} \lambda_n^{-\frac12}[\chi_nw_n](\lambda_n^{-2}t, \lambda_n^{-1}(x-x_n)), & |t|\le \lambda_n^2 T, \\
e^{-i(t-\lambda_n^2 T)\ld}\tilde v_n(\lambda_n^2 T,x), & t>\lambda_n^2 T, \\
e^{-i(t+\lambda_n^2 T)\ld}\tilde v_n(-\lambda_n^2 T,x), & t<-\lambda_n^2 T,
\end{cases}
\end{align*}
where $\chi_n$ is a family of smooth functions satisfying
\begin{equation*}
 |\partial^k \chi_n(x)| \lesssim \bigl( \tfrac{\lambda_n}{|x_n|} \bigr)^k
\qtq{and}
\chi_n(x)=
\begin{cases}
0 \quad&\text{if}\quad   |x_n+\lambda_n x| \leq\tfrac{|x_n|}4,\\
1 \quad&\text{if}\quad  |x_n+\lambda_n x| \geq\tfrac{|x_n|}2.
\end{cases}
\end{equation*}

By a change of variables, H\"older, and \eqref{cond3},
$$
\|\chi_n w_n(\pm T)\|_{\dot H^1_x}\lesssim \|\nabla\chi_n\|_{L_x^3} \|w_n(\pm T)\|_{L_x^6} + \|\chi_n\|_{L_x^\infty} \|\nabla w_n(\pm T)\|_{L_x^2}\lesssim 1.
$$
Combining this with the Strichartz inequality and Theorem~\ref{pro:equivsobolev}, we get
\begin{align}\label{tildevn3}
\|\sqrt{\ld} \tilde v_n\|_{L_t^{10}L_x^{\frac{30}{13}} \cap L_t^\infty L_x^2}
&\lesssim \| \sqrt{\ld^n} (\chi_n w_n)\|_{L_t^{10}L_x^{\frac{30}{13}} \cap L_t^\infty L_x^2} +\|\chi_n w_n(\pm T)\|_{\dot H^1_x}\notag\\
&\lesssim \| \nabla(\chi_n w_n)\|_{L_t^{10}L_x^{\frac{30}{13}} \cap L_t^\infty L_x^2} + 1\notag\\
&\lesssim \|w_n\|_{\dot S^1} + 1\lesssim 1.
\end{align}

\textbf{Step 3:}  Asymptotic agreement of the initial data:
\begin{align}\label{n0}
\lim_{T\to\infty}\limsup_{n\to\infty}\|\sqrt{\ld}[\tilde v_n(\lambda_n^2t_n)-\phi_n]\|_{L_x^2}=0.
\end{align}

We first consider the case when $t_n\equiv0$.  Using the equivalence of Sobolev spaces and performing a change of variables, we obtain
\begin{align*}
\|\sqrt{\ld}[\tilde v_n(0)-\phi_n]\|_{L_x^2}
&\lesssim \|\nabla[\chi_n\phi_{>(|x_n|/\lambda_n)^{\theta}}]\|_{L^2_x} + \| \nabla[(1-\chi_n)\phi]\|_{L_x^2}\\
&\lesssim \|\nabla \chi_n\|_{L_x^3}\|\phi_{>(|x_n|/\lambda_n)^{\theta}}\|_{L_x^6}+\|\nabla \phi_{>(|x_n|/\lambda_n)^{\theta}}\|_{L_x^2}\\
&\quad+  \|\nabla \chi_n\|_{L_x^3} \|\phi\|_{L_x^6(\supp(\nabla \chi_n))} + \|\nabla \phi\|_{L_x^2(\supp(1-\chi_n))} ,
\end{align*}
which converges to zero as $n\to \infty$ since $\chi_n(x)\to 1$ and $\chi_{\supp(\nabla\chi_n)}(x)\to 0$ almost everywhere as $n\to\infty$.

It remains to prove \eqref{n0} when $t_n\to \infty$; the case $t_n\to-\infty$ can be treated analogously.  As $T$ is fixed, for sufficiently large $n$ we have $t_n>T$ and so
\begin{align*}
\tilde v_n(\lambda_n^2t_n,x)=e^{-i(t_n-T)\lambda_n^2\ld}\bigl[\lambda_n^{-\frac 12}(\chi_nw_n(T))\bigl(\tfrac{x-x_n}{\lambda_n}\bigr)\bigr].
\end{align*}
Thus, by a change of variables and the equivalence of Sobolev spaces,
\begin{align}
\| \tilde v_n(\lambda_n^2 t_n)-\phi_n \|_{\dot H^1_a} &=\|\sqrt {\ld^n} [e^{iT\ld^n}(\chi_nw_n(T))-\phi]\|_{L_x^2}\notag\\
&\lesssim \|\nabla[\chi_n(w_n(T)-w_\infty(T))]\|_{L^2_x}\label{n1}\\
&\quad + \|\nabla[\chi_n w_\infty(T) - w_\infty(T)]\|_{L^2_x}\label{n2}\\
&\quad + \|e^{iT\ld^n} \sqrt {\ld^n} \, w_\infty(T) - \sqrt {\ld^n} \,\phi\|_{L_x^2}\label{n3},
\end{align}
where $\ld^n$ is as in Definition~\ref{D:Ln} and $y_n=x_n/\lambda_n\to \infty$.  Note that in this case $\ld^\infty=-\Delta$.

Using \eqref{cond3} and Sobolev embedding, we see that
\begin{align*}
\eqref{n1}\lesssim \|\nabla\chi_n\|_{L_x^3}\|w_n(T)-w_\infty(T)\|_{L_x^6}+\|\chi_n\|_{L_x^\infty}\|\nabla[w_n(T)-w_\infty(T)]\|_{L_x^2} \to 0
\end{align*}
as $n\to \infty$.

By H\"older and the dominated convergence theorem,
\begin{align*}
\eqref{n2}\lesssim \|\nabla\chi_n\|_{L_x^3} \|w_\infty(T)\|_{L_x^6(\supp(\nabla\chi_n))} + \|(1-\chi_n)\nabla w_\infty(T)]\|_{L_x^2} \to 0
\end{align*}
as $n\to\infty$.

Now we look to \eqref{n3}.  By \eqref{convg op5} we have
$$
\bigl\|\bigl[\sqrt {\ld^n} -\sqrt {\ld^\infty}\,\bigr] w_\infty(T) \bigr\|_{L_x^2} + \bigl\|\bigl[\sqrt {\ld^n} -\sqrt {\ld^\infty}\,\bigr] \phi \bigr\|_{L_x^2} \to 0 \qtq{as} n\to \infty,
$$
while by \eqref{convg op4} we have
$$
\bigl\|\bigl[e^{iT\ld^n} -e^{iT\ld^\infty}  \bigr]\sqrt {\ld^\infty}\, w_\infty(T) \bigr\|_{L_x^2}  \to 0 \qtq{as} n\to \infty.
$$
Combining these, we deduce that 
$$
\limsup_{n\to\infty} \eqref{n3} =  \|\sqrt {\ld^\infty} [e^{iT\ld^\infty} w_\infty(T) - \phi]\|_{L_x^2},
$$
which converges to zero as $T\to\infty$ by the construction of  $w_\infty$.  This completes the proof of \eqref{n0}.

\textbf{Step 4:}  We prove that $\tilde v_n$ is an approximate solution to \eqref{equ1.1} in the sense required by Theorem~\ref{thm:stability}:
\begin{align}\label{n6}
\lim_{T\to\infty}\limsup_{n\to\infty}\|\sqrt{\ld} \,[(i\partial_t-\ld )\tilde v_n-|\tilde v_n|^4\tilde v_n]\|_{N^0(\R)}=0.
\end{align}

We first verify \eqref{n6} for $|t|>\lambda_n^2 T$.  By symmetry, it suffices to consider positive times.  For these times we have
$$
e_n:=(i\partial_t-\ld )\tilde v_n-|\tilde v_n|^4\tilde v_n = -|\tilde v_n|^4\tilde v_n.
$$
By Lemma~\ref{L:Leibnitz} and \eqref{tildevn3},
\begin{align*}
\|\sqrt{\ld}e_n\|_{L_t^2L_x^{\frac65}(\{t>\lambda_n^2 T\}\times\R^3)}
&\lesssim \|\sqrt{\ld}\tilde v_n\|_{L_t^{10}L_x^{\frac{30}{13}}(\{t>\lambda_n^2 T\}\times\R^3)}\|\tilde v_n\|_{L_{t,x}^{10}(\{t>\lambda_n^2 T\}\times\R^3)}^4\\
&\lesssim \|e^{-it\ld^n} (\chi_nw_n(T))\|_{L_{t,x}^{10}((0, \infty)\times\R^3)}^4.
\end{align*}
On the other hand, by the analysis of \eqref{n1} and \eqref{n2}, we have
$$
\|e^{-it\ld^n} (\chi_nw_n(T))\|_{L_{t,x}^{10}((0, \infty)\times\R^3)} = \|e^{-it\ld^n} w_\infty(T)\|_{L_{t,x}^{10}((0, \infty)\times\R^3)} +  o(1)
$$
as $n\to\infty$.

To continue, let $w_+$ denote the forward asymptotic state of $w_\infty$; its existence is guaranteed by Theorem~\ref{T:gopher}.
By Strichartz,
\begin{align*}
\|e^{-it\ld^n} &w_\infty(T)\|_{L_{t,x}^{10}((0, \infty)\times\R^3)}\\
&\lesssim \|[e^{-it\ld^n}-e^{-it\ld^\infty}] w_\infty(T)\|_{L_{t,x}^{10}((0,\infty)\times\R^3)}+\|w_\infty(T)-e^{iT\Delta}w_+]\|_{\dot H^1(\R^3)}\\
&\quad+ \|e^{it\Delta}w_+\|_{L_{t,x}^{10}((T, \infty)\times\R^3)},
\end{align*}
which converges to zero as $n\to \infty$ and then $T\to \infty$ in view of Corollary~\ref{C:L10}, the definition of $w_+$, and the monotone convergence theorem.

Next we show \eqref{n6} on the middle time interval $|t|\le \lambda_n^2 T$.  For these times,
\begin{align}
e_n(t,x)&:=[(i\partial_t-\ld)\tilde v_n-|\tilde v_n|^4\tilde v_n](t,x)\notag\\
&=\lambda_n^{-\frac 52}[(\chi_n-\chi_n^5)|w_n|^4w_n](\lambda_n^{-2}t,\lambda_n^{-1}(x-x_n))\label{149}\\
&\quad+2\lambda_n^{-\frac 52}[\nabla\chi_n \cdot\nabla w_n](\lambda_n^{-2} t, \lambda_n^{-1}(x-x_n))\label{150}\\
&\quad+\lambda_n^{-\frac 52}[\Delta\chi_nw_n](\lambda_n^{-2} t,\lambda_n^{-1}(x-x_n))\label{151}\\
&\quad -\lambda_n^{-\frac 12}\tfrac{a}{|x|^2}[\chi_nw_n](\lambda_n^{-2} t,\lambda_n^{-1}(x-x_n)).\label{152}
\end{align}
Performing a change of variables and using Lemma~\ref{L:Leibnitz} and H\"older, we estimate the contribution of \eqref{149} by
\begin{align*}
&\|(\chi_n-\chi_n^5)|w_n|^4 \nabla w_n\|_{L_t^2L_x^{\frac65}}+\|\nabla \chi_n(1-5\chi_n^4)w_n^5\|_{L_t^2L_x^{\frac65}}\\
&\lesssim\Bigl[\|\nabla w_n\|_{L_t^{10}L_x^{\frac{30}{13}}}+\|w_n\|_{L_{t,x}^{10}}\|\nabla \chi_n\|_{L_x^3}\Bigr]\Bigl[\|w_n-w_\infty\|_{L_{t,x}^{10}}^4+\|1_{|x|\sim \frac{|x_n|}{\lambda_n}} w_\infty\|_{L_{t,x}^{10}}^4\Bigr],
\end{align*}
which converges to zero as $n\to \infty$ by \eqref{cond3} and the dominated convergence theorem.  Similarly, we estimate the contribution of \eqref{150} and \eqref{151} by
\begin{align*}
T \Bigl[ \|\nabla \chi_n\|_{L_x^\infty}\|\Delta w_n&\|_{L_t^\infty L_x^2}  + \|\Delta\chi_n\|_{L_x^\infty} \|\nabla w_n\|_{L_t^\infty L_x^2} + \|\nabla\Delta\chi_n\|_{L_x^3}\|w_n\|_{L_t^\infty L_x^6}\Bigr]\\
&\lesssim T\Bigl[\bigl(\tfrac{\lambda_n}{|x_n|}\bigr)^{1-\theta} + \bigl(\tfrac{\lambda_n}{|x_n|}\bigr)^{2}+ \bigl(\tfrac{\lambda_n}{|x_n|}\bigr)^{3}\Bigr] \to 0 \qtq{as} n\to \infty.
\end{align*}
Lastly, we estimate the contribution of \eqref{152} by
\begin{align*}
T &\bigl\|\nabla\bigl(\tfrac{\lambda_n^2\chi_n(x)}{|x_n+\lambda_n x|^2}w_n\bigr)\bigr\|_{L_t^\infty L_x^2}\\
&\lesssim T \Bigl[\bigl\|\nabla\tfrac{\lambda_n^2\chi_n(x)}{|x_n+\lambda_n x|^2}\bigr\|_{L_x^3} \|w_n\|_{L_t^\infty L_x^6} +\bigl\|\tfrac{\lambda_n^2\chi_n(x)}{|x_n+\lambda_n x|^2}\bigr\|_{L_x^\infty} \|\nabla w_n\|_{L_t^\infty L_x^2} \Bigr]\\
&\lesssim T \bigl(\tfrac{\lambda_n}{|x_n|}\bigr)^2 \to 0 \qtq{as} n\to \infty.
\end{align*}
This completes the proof of \eqref{n6}.

\textbf{Step 5:} Constructing $v_n$ and approximation by $C_c^\infty$ functions.

Using \eqref{tildevn3}, \eqref{n0}, \eqref{n6}, and Theorem~\ref{thm:stability}, for $n$ sufficiently large we obtain a global solution $v_n$ to \eqref{equ1.1} with initial data $v_n(0)=\phi_n$ which satisfies
\begin{align*}
\|v_n\|_{L_{t,x}^{10}(\R\times\R^3)}\lesssim 1
\qtq{and}
\lim_{T\to\infty}\limsup_{n\to \infty}\|v_n(t-\lambda_n^2t_n)-\tilde v_n (t)\|_{\dot S^1(\R\times\R^3)}=0.
\end{align*}
It remains to prove the approximation result \eqref{dense2}.

From the density of $C_c^\infty(\R\times\R^3)$ in $\dot X^1(\R)$, for any $\eps>0$ there exists $\psi_\eps\in C_c^\infty(\R\times\R^3)$ such that
\begin{align*}
\|w_\infty-\psi_\eps\|_{\dot X^1(\R)}< \tfrac \eps 3.
\end{align*}
Thus, proving \eqref{dense2} reduces to showing
\begin{align}\label{n11}
\|\tilde v_n(t,x)-\lambda_n^{-\frac 12}w_\infty(\lambda_n^{-2}t, \lambda_n^{-1}(x-x_n))\|_{\dot X^1(\R)}<\tfrac \eps 3
\end{align}
for $n, T$ sufficiently large.  Changing variables we estimate
\begin{align*}
\text{LHS}\eqref{n11}
&\le \|\chi_n w_n-w_\infty\|_{\dot X^1([-T,T])}+\|e^{-i(t-T)\ld^n}[\chi_nw_n(T)]-w_\infty\|_{\dot X^1((T,\infty))}\\
&\quad+\|e^{-i(t+T)\ld^n}[\chi_nw_n(-T)]-w_\infty\|_{\dot X^1((-\infty,-T))}.
\end{align*}
We consider each of these three terms separately.  Using the dominated convergence theorem and \eqref{cond3}, we see that
\begin{align}\label{1018}
\|\chi_n w_n-w_\infty\|_{\dot X^1([-T,T])}&\lesssim \|(1-\chi_n)w_\infty\|_{\dot X^1(\R)}+\|w_n-w_\infty\|_{\dot X^1(\R)}\to 0
\end{align}
as $n\to \infty$. The second and third term can be treated similarly; we only present the details for the second term.  By Strichartz, 
\begin{align*}
\|&e^{-i(t-T)\ld^n}[\chi_nw_n(T)]-w_\infty\|_{\dot X^1((T,\infty))}\\
&\lesssim \|w_\infty\|_{\dot X^1((T,\infty))}+\bigl\|\sqrt {\ld^n}\bigl[\chi_nw_n(T)-w_\infty(T)\bigr]\bigr\|_{L^2_x} +\bigl\|\bigl[\sqrt {\ld^n} -\sqrt {\ld^\infty}\,\bigr] w_\infty(T) \bigr\|_{L_x^2}\\
&\quad+\bigl\|\bigl[e^{-i(t-T)\ld^n}-e^{i(t-T)\Delta}\bigr]\sqrt {\ld^\infty}\, w_\infty(T)\bigr\|_{L_t^{10}L_x^{\frac{30}{13}}}+\|e^{it\Delta}w_+\|_{\dot X^1((T,\infty))}\\
&\quad+\|e^{-iT\Delta}w_\infty(T)-w_+\|_{\dot H^1_x}\to 0 \quad\text{as $n\to \infty$ and then $T\to \infty$}
\end{align*}
by the monotone convergence theorem, \eqref{1018}, Lemma~\ref{L:convg op}, and the definition of the asymptotic state $w_+$.
This completes the proof of \eqref{n11} and with it, the proof of Theorem~\ref{T:embed}.
\end{proof}




\section{Existence of the minimal blowup solution}\label{S:5}

The main goal of this section is to show that the failure of Theorem~\ref{theorem} implies the existence of a minimal counterexample that has good compactness properties.  Note that the failure of Theorem~\ref{theorem} implies the existence of a critical energy $E_c$ such that \eqref{goal} holds for solutions with energy less than $E_c$, but it fails for solutions with energy greater than $E_c$.  The strict positivity of the critical energy $E_c$ is guaranteed by the small data global well-posedness theory.  Observe also that finiteness of $E_c$ is equivalent to the failure of Theorem~\ref{theorem}.

\begin{theorem}[Existence of almost periodic solutions]\label{T:mmbs}
Suppose Theorem~\ref{theorem}  fails. Then there exist a critical energy $0<E_c<\infty$ and a solution $u:[0, T^*)\times\R^3\to \C$ to \eqref{equ1.1} with
$$
E(u)=E_c\qtq{and} \|u\|_{L_{t,x}^{10}([0, T^*)\times\R^3)}=\infty.
$$
Moreover, the orbit of $u$ is precompact in $\dot H^1(\R^3)$ modulo scaling in the sense that there exists $N(t): I\to\R^+$ such that the set $\big\{
N(t)^{-\frac12}u\big(t,\frac{x}{N(t)}\big):\, t\in [0, T^*)\}$ is precompact in $\dot H^1(\R^3)$.  Finally, we may assume that the frequency scale function $N(t)$ satisfies $\inf_{t\in [0, T^*)}N(t)\geq1$.
\end{theorem}

Results of this type are by now standard in the literature.  The key result underlying the existence of the minimal counterexample (with critical energy and infinite $L_{t,x}^{10}$-norm) is a Palais--Smale condition for minimizing sequences of blowup solutions to \eqref{equ1.1}; this is recorded as Proposition~\ref{P:PS} below.  This Palais--Smale condition also implies that the orbit of the minimal counterexample is precompact modulo scaling.  Finally, having extracted this minimal blowup solution, a simple rescaling argument allows us to assume that the frequency scale function $N(t)$ is bounded from below either in the future or in the past; see for example \cite[\S4]{KV5}.  The interval $[0, T^*)$ appearing in Theorem~\ref{T:mmbs} is the maximal future extension of $u$, where we can guarantee that $\inf_{t\in [0, T^*)}N(t)\geq1$.

\begin{proposition}[Palais--Smale condition]\label{P:PS}
Let $u_n: I_n\times\R^3\to \C$ be a sequence of solutions with $E(u_n)\to E_c$ and let $t_n\in I_n$ so that
\begin{align*}
\lim_{n\to\infty} \|u_n\|_{L_{t,x}^{10}(\{t\geq t_n\}\times\R^3)}=\lim_{n\to\infty}\|u_n\|_{L_{t,x}^{10}(\{t \leq t_n\}\times\R^3)}=\infty.
\end{align*}
Then, passing to a subsequence, $\{u(t_n)\}$ converges in $\dot H^1(\R^3)$ modulo scaling.
\end{proposition}

\begin{proof}
The proof of Proposition~\ref{P:PS} follows along well established lines.  For a thorough discussion of this result and its relevance in proving Theorem~\ref{T:mmbs} in the absence of a potential see \cite{KVnote}.  The presence of the potential introduces several new difficulties.  These are of a similar nature as those arising for the energy-critical NLS outside a convex obstacle and will be overcome by mimicking the arguments in \cite{KVZ12}.  In what follows, we will sketch the proof emphasizing the main steps.

By time translation symmetry, we may assume $t_n\equiv0$; thus,
\begin{align}\label{scat diverge}
\lim_{n\to\infty} \|u_n\|_{L_{t,x}^{10}(\{t\geq 0\}\times\R^3)}=\lim_{n\to\infty}\|u_n\|_{L_{t,x}^{10}(\{t \leq 0\}\times\R^3)}=\infty.
\end{align}
Applying Proposition~\ref{T:LPD} to the sequence $u_n(0)$ (which is bounded in $\dot H^1_a(\R^3)$) and passing to a subsequence if necessary, we have the linear profile decomposition
\begin{align}\label{s0}
u_n(0)=\sum_{j=1}^J \phi_n^j+w_n^J
\end{align}
satisfying the properties stated in Proposition~\ref{T:LPD}. In particular, for any finite $0\leq J \leq J^*$, we have the energy decoupling property
\begin{align}\label{s01}
\lim_{n\to \infty}\Bigl\{E(u_n)-\sum_{j=1}^J E(\phi_n^j)-E(w_n^J)\Bigr\}=0.
\end{align}

To prove Proposition~\ref{P:PS}, we need to show that $J^*=1$, $w_n^1\to 0$ in $\dot H^1(\R^3)$, $t_n^1\equiv 0$, and $x_n\equiv 0$.  To this end, we will show that all other possibilities contradict \eqref{scat diverge}.  We discuss two cases.

\textbf{Case I:} $\sup_j \limsup_{n\to \infty} E(\phi_n^j)=E_c$.

From the Hardy inequality and the non-triviality of the profiles, we deduce that $\liminf_{n\to \infty} E(\phi_n^j)>0$ for each finite $1\leq j\leq J^*$. Thus, \eqref{s01} implies that there is a single profile in the decomposition \eqref{s0} (that is, $J^*=1$) and we can write
\begin{equation}\label{s11}
u_n(0)=\phi_n +w_n \quad\text{with}~~ \lim_{n\to \infty} \|w_n\|_{\dot H^1(\R^3)}=0.
\end{equation}

If $\frac{|x_n|}{\lambda_n}\to \infty$, then for $n$ sufficiently large Theorem~\ref{T:embed} guarantees the existence of a global solution $v_n$ to \eqref{equ1.1} with initial data $v_n(0)=\phi_n$ and uniform bounded $L_{t,x}^{10}$-norm.  By Theorem~\ref{thm:stability}, this space-time bound extends to the solution $u_n$ for $n$ sufficiently large, thus contradicting \eqref{scat diverge}.  Therefore, we must have $x_n\equiv 0$.

To obtain the desired compactness property, it remains to preclude the case $t_n\to\pm\infty$.  By symmetry, it suffices to consider $t_n\to \infty$.  By Strichartz, monotone convergence, and Corollary~\ref{C:L10}, we see that 
\begin{align*}
\|e^{-it\ld}u_n(0)\|_{L_{t,x}^{10}(\{t>0\}\times\R^3)}
&\leq \|e^{-it\ld}w_n\|_{L_{t,x}^{10}(\{t>0\}\times\R^3)} + \|e^{-it\ld^\infty}\phi\|_{L_{t,x}^{10}(\{t>t_n\}\times\R^3)}\\
&\quad +\|[e^{-it\ld^n}-e^{-it\ld^\infty}]\phi\|_{L_{t,x}^{10}(\R\times\R^3)} \to 0 \qtq{as} n\to \infty.
\end{align*}
By the small data theory, this implies that $\|u_n\|_{L_{t,x}^{10}(\{t>0\}\times\R^3)}\to 0$ as $n\to \infty$, which again contradicts \eqref{scat diverge}.

\medskip

\textbf{Case II:} $\sup_j \limsup_{n\to \infty} E(\phi_n^j) \leq E_c-3\delta$ for some $\delta>0$.

In this case, for each finite $J\leq J^*$ we have $E(\phi_n^j) \leq E_c-2\delta$ for all $1\leq j\leq J$ and $n$ sufficiently large.  

We will prove that Case~II is inconsistent with \eqref{scat diverge}. To this end, we first introduce nonlinear profiles associated with each $\phi_n^j$.

For those $j$ for which $\frac{|x_n^j|}{\lambda_n^j}\to+\infty$, let $v_n^j$ denote the global solution to \eqref{equ1.1} guaranteed by Theorem~\ref{T:embed}.

We next consider those $j$ for which $x_n^j\equiv 0$.  If $t_n^j\equiv 0$, let $v^j$ denote the maximal-lifespan solution to \eqref{equ1.1} with initial data $v^j(0)=\phi^j$.  If instead $t_n^j\to \pm \infty$, let $v^j$ denote the maximal-lifespan solution to \eqref{equ1.1} which scatters to $e^{-it\ld}\phi^j$ as $t\to \pm\infty$.  In both cases we define
$$
v_n^j(t,x):=(\lambda_n^j)^{-\frac12}v^j\big(\tfrac{t}{(\lambda_n^j)^2}+t_n^j,\tfrac{x}{\lambda_n^j}\big).
$$
Note that $v_n^j$ is also a solution to \eqref{equ1.1} with
\begin{align*}
\lim_{n\to\infty}\|v_n^j(0)-\phi_n^j\|_{\dot H^1(\R^3)}=0.
\end{align*}
In particular, by the Hardy inequality we have that $E(v_n^j) \leq E_c-\delta$ for all $1\leq j\leq J$ and $n$ sufficiently large.  By the definition of the critical energy $E_c$, this implies that the solutions $v_n^j$ are global in time and have finite Strichartz norms.

The next step in the argument is to construct an approximate solution to \eqref{equ1.1} with finite $L_{t,x}^{10}$-norm. We define
\begin{align*}
u_n^J:=\sum_{j=1}^J v_n^j+e^{-it\ld}w_n^J.
\end{align*}
By construction, $u_n^J$ is defined globally in time and satisfies
$$
\|u_n^J(0)-u_n(0)\|_{\dot H^1(\R^3)}\to 0\qtq{as} n\to\infty
$$
for any $J$.

The asymptotic decoupling of parameters \eqref{E:LP5} begets (as in \cite[Lemma~7.3]{KVZ12}) asymptotic decoupling of the nonlinear solutions $v_n^j$.  Using this together with the control we have over the $L_t^{10}\dot H^{1, \frac{30}{13}}_x$-norm of each $v_n^j$, one can show that
$$
\limsup_{n\to \infty} \|u_n^J\|_{L_t^{10}\dot H^{1, \frac{30}{13}}_x(\R\times\R^3)}\lesssim_{E_c, \delta} 1,
$$
uniformly in $J$.

Lastly, one can prove that for large enough $n$ and $J$, $u_n^J$ is an approximate solution to \eqref{equ1.1} in the sense that
$$
\lim_{J\to\infty}\limsup_{n\to\infty}\bigl\|\sqrt{\ld}\bigl[(i\partial_t-\ld)u_n^J-|u_n^J|^4u_n^J\bigr]\bigr\|_{N^0(\R)}=0.
$$
In order to verify this statement, one uses the asymptotic decoupling of the solutions $v_n^j$ together with the fact that, by the linear profile decomposition, $e^{-it\ld}w_n^J$ converges to zero (as $n$ and $J$ converge to infinity) in energy-critical Strichartz spaces without derivatives.  Corollary~\ref{C:Keraani3.7} is used to control terms involving one derivative of $e^{-it\ld}w_n^J$; see, for example, the justification of (7.19) in \cite{KVZ12}.

Putting everything together and invoking Theorem~\ref{thm:stability}, we deduce that the solutions $u_n$ inherit the space-time bounds of $u_n^J$ for $n$ sufficiently large.  This however contradicts \eqref{scat diverge}.
\end{proof}




\section{Precluding the minimal blowup solution}

Let us suppose that Theorem~\ref{theorem} were to fail and let $u:[0, T^*)\times\R^3\to \C$ be a minimal counterexample of the type provided by Theorem~\ref{T:mmbs}.  The fact that the orbit of $u$ is precompact in $\dot H^1_x$ modulo scaling combined with the lower bound $N(t)\geq 1$ on the frequency scale function yields that for any $\eta>0$ there exists $C(\eta)>0$ such that
\begin{align}\label{compactness}
\int_{|x|\geq C(\eta)} |\nabla u(t,x)|^2 + |u(t,x)|^6\, dx\leq \int_{|x|\geq \frac{C(\eta)}{N(t)}} |\nabla u(t,x)|^2 + |u(t,x)|^6\, dx\leq \eta,
\end{align}
uniformly for $t\in [0, T^*)$.
 
Note also that by the Hardy inequality and the conservation of energy, for all solutions $u$ to \eqref{equ1.1} we have
\begin{align}\label{ke}
\|\nabla u(t)\|_{L_x^2}^2\sim E(u) \quad\text{uniformly for $t\in [0, T^*)$}.
\end{align}

To preclude the existence of the minimal counterexample $u$ guaranteed by Theorem~\ref{T:mmbs}, we will distinguish two cases, namely, (1) the solution $u$ is global forward in time (i.e. $T^*=\infty$) or (2) the solution $u$ blows up in finite time (i.e. $T^*<\infty$).  We will show that the first scenario is inconsistent with the virial identity; see Theorem~\ref{T:7.1}.  Finally, we will employ a transport of mass argument to preclude the second scenario; see Theorem~\ref{T:7.2}.  This will complete the proof of Theorem~\ref{theorem}.

\begin{theorem}\label{T:7.1}
There are no solutions to \eqref{equ1.1} of the form given in Theorem~\ref{T:mmbs} with $T^*=\infty$.
\end{theorem}

\begin{proof}
Assume towards a contradiction that $u:[0,\infty)\times\R^3\to \C$ is such a solution.  Let $\phi$ be a smooth radial cutoff such that
\begin{align*}
\phi(r)=\begin{cases}
r & \text{for } r\leq 1\\
0 & \text{for } r\geq 2,
\end{cases}
\end{align*}
and define
$$
V_R(t): =\int_{\R^3} \psi(x) |u(t,x)|^2\,dx \qtq{where} \psi(x):=R^2\phi\bigl(\tfrac{|x|^2}{R^2}\bigr)
$$
for some $R>0$.

Differentiating $V_R$ with respect to the time variable and using H\"older, Sobolev embedding, and \eqref{ke}, we find
\begin{align}\label{1 deriv}
|\partial_t V_R(t)| &= \Bigl|4\Im \int_{\R^3} \phi'\bigl(\tfrac{|x|^2}{R^2}\bigr) \overline{u(t,x)} \,x\cdot \nabla u(t,x) \,dx\Bigr|\notag\\
&\lesssim R^2  \|u(t)\|_6\|\nabla u(t)\|_2\lesssim_u R^2,
\end{align}
uniformly for $t\geq 0$.

Taking another derivative with respect to the time variable and using the Hardy inequality, Sobolev embedding, and \eqref{compactness}, we obtain
\begin{align*}
\partial_{tt}V_R(t) &= 4 \Re \int_{\R^3}\psi_{ij}(x)u_{i}(t,x)\bar u_j(t,x)\, dx
        - \tfrac 43 \int_{\R^3} \bigl(\Delta \psi\bigr)(x) |u(t,x)|^{6}\, dx\\
&\quad  - \int_{\R^3} \bigl(\Delta \Delta \psi\bigr)(x) |u(t,x)|^2\, dx - 4a\int_{\R^3}  \tfrac{x}{|x|^4}\nabla \psi(x)|u(t,x)|^2\, dx\\
&=8\int_{\R^3} |\nabla u(t,x)|^2 + \tfrac a{|x|^2}|u(t,x)|^2 + |u(t,x)|^6\, dx\\
&\quad + O\Bigl(\int_{|x|\geq R} |\nabla u(t,x)|^2 + |u(t,x)|^6\, dx\Bigr) \\
&\gtrsim \|\nabla u(t)\|_{L_x^2}^2 - \eta,
\end{align*}
provided $\eta$ is small and $R=R(\eta)$ is chosen sufficiently large.  Combining this with \eqref{ke} and taking $\eta$ sufficiently small depending on the energy of $u$, we get
$$
\partial_{tt}V_R(t) \gtrsim_u 1\quad\text{uniformly for $t\geq 0$.}
$$
Together with \eqref{1 deriv} and the fundamental theorem of calculus on $[0, T]$ for $T$ sufficiently large, this yields the desired contradiction.
\end{proof}

\begin{theorem}\label{T:7.2}
There are no solutions to \eqref{equ1.1} of the form given in Theorem~\ref{T:mmbs} with $T^*<\infty$.
\end{theorem}

\begin{proof}
Assume towards a contradiction that $u:[0, T^*)\times\R^3\to \C$ is such a solution.  By Corollary~5.19 in \cite{KVnote}, this implies
\begin{equation}\label{eq8.3}
\liminf_{t\nearrow T^*} N(t)=\infty.
\end{equation}
Consequently, the mass of $u$ leaves any fixed ball as $t\to T^*$; specifically,
\begin{equation}\label{eq8.9}
\limsup_{t\nearrow T^*}\int_{|x|\leq R}|u(t,x)|^2\,dx=0\quad\text{for any $R>0$.}
\end{equation}
Indeed, for $0<\eta<1$ and $t\in [0, T^*)$, using H\"older we estimate
\begin{align*}
\int_{|x|\leq R}|u(t,x)|^2\, dx
&\leq\int_{|x|\leq \eta R}|u(t,x)|^2\, dx+\int_{\eta R\leq |x|\leq R}|u(t,x)|^2\,dx\\
&\lesssim \eta^2R^2\|u(t)\|_{L_x^6}^2+R^2\Big(\int_{|x|\geq \eta R}|u(t,x)|^{6}\,dx\Big)^{1/3}.
\end{align*}
By \eqref{ke} and Sobolev embedding, we can make the first term arbitrarily small by letting $\eta\to 0$.  Moreover, by \eqref{compactness} and \eqref{eq8.3}, we see that the second term converges to zero as $t\to T^*$.

To continue, for $t\in [0,T^*)$ we define
\begin{equation}\label{eq8.13}
M_{R}(t):=\int_{\mathbb{R}^3}\phi\big(\tfrac{|x|}{R}\big)|u(t,x)|^2\,dx,
\end{equation}
where $\phi$ is a smooth radial function such that $\phi(r)=1$ for $r\leq 1$ and $\phi(r)=0$ for $r\geq 2$.  By \eqref{eq8.9},
\begin{equation}\label{eq8.14}
\limsup_{t\nearrow T^*} M_{R}(t)=0\quad\text{for all $R>0$.}
\end{equation}
On the other hand, by the Hardy inequality and \eqref{ke},
\begin{equation*}
|\partial_{t}M_{R}(t)|\lesssim\|\nabla u(t)\|_{L^2_x}\big\|\tfrac{u(t)}{|x|}\big\|_{L^2_x}\lesssim_u1.
\end{equation*}
Thus, by the fundamental theorem of calculus,
\begin{equation*}
M_{R}(t_1)=M_{R}(t_2)+\int_{t_2}^{t_1}\partial_{t}M_{R}(\tau)\,d\tau\lesssim_u M_{R}(t_2)+|t_1-t_2|
\end{equation*}
for all $t_1,t_2\in [0,T^*)$ and $R>0$.  Letting $t_2\nearrow T^*$, we deduce that
\begin{equation*}
M_{R}(t_1)\lesssim_u\big|T^*-t_1\big|.
\end{equation*}
Now letting $R\to\infty$ and invoking the conservation of mass, we derive
\begin{equation*}
\|u(0)\|_{L^2_x}^2\lesssim_u \big|T^*-t_1\big|.
\end{equation*}
Finally, letting $t_1\nearrow T^*$ we obtain $u(0)\equiv0$, which contradicts the fact that the $L_{t,x}^{10}$-norm of $u$ is infinite.
\end{proof}

\section{The focusing case}\label{S:7}

In this section, we consider the analogue of \eqref{equ1.1} with focusing nonlinearity, posed in $\R^d$ with $d\geq 3$:
\begin{equation} \label{E:fNLS}
 (i\partial_t-\ld)u= - |u|^\frac{4}{d-2}u \qtq{with} u(t=0) \in \dot H^1_a(\R^d),
\end{equation}
which conserves the energy
\begin{equation}\label{energy_a}
E_a(f):=\int_{\R^d}  \tfrac12|\nabla f(x)|^2 + \tfrac{a}{2|x|^2} |f(x)|^2 - \tfrac{d-2}{2d} |f(x)|^{\frac{2d}{d-2}}\,dx.
\end{equation}
Note that for the topics we will be discussing in this section, it costs very little in clarity to treat general dimensions.

\begin{definition}  Given $a>-(\frac{d-2}2)^2$, we define $\beta>0$ via $a=(\frac{d-2}2)^2[\beta^2-1]$, or equivalently, $\sigma=\frac{d-2}2(1-\beta)$.  We then define the \emph{ground state soliton}  by
\begin{equation}\label{E:Wa}
W_a(x) :=    [d(d-2)\beta^2]^{\frac{d-2}{4}} \biggl[ \frac{ |x|^{\beta-1} }{ 1+|x|^{2\beta} }\biggr]^{\frac{d-2}{2}}.
\end{equation}
\end{definition}

It is not difficult (though laborious) to verify that 
$$
\ld W_a = |W_a|^\frac{4}{d-2} W_a 
$$
and, using a standard variant of Euler's Beta integral (cf. (1.1.20) in \cite{AAR}), that
\begin{equation}\label{WaPoho}
\|W_a\|_{\dot H^{1}_a(\R^d)}^2 = \int_{\R^d} |W_a(x)|^{\frac{2d}{d-2}}\,dx = 
	\tfrac{\pi d(d-2)}{4} \Bigl[\tfrac{2\sqrt{\pi}\beta^{d-1}}{\Gamma(\frac{d+1}{2})} \Bigr]^\frac2d.
\end{equation}
Thus $W_a$ is a ground state soliton in the sense of being a radial non-negative static solution to \eqref{E:fNLS}.   As we will see in Proposition~\ref{P:Wa} below,  these solitons occur as optimizers in Sobolev embedding inequalities when $a\leq 0$, but not when $a>0$.  The proof of that proposition will also explain how we derived formula \eqref{E:Wa} .

\begin{proposition}[Sharp Sobolev embedding] \label{P:Wa}
Fix $d\geq 3$ and $a>-(\frac{d-2}2)^2 $.\\
{\upshape(i)} If $-(\frac{d-2}2)^2 < a < 0$, then
\begin{equation}\label{E:SS-}
\|f\|_{L^{\frac{2d}{d-2}}_x(\R^d)} \leq \|W_a\|_{L^{\frac{2d}{d-2}}_x(\R^d)} \|W_a\|_{\dot H^{1}_a(\R^d)}^{-1} \|f\|_{\dot H^{1}_a(\R^d)}^{ }. 
\end{equation}
Moreover, equality holds in \eqref{E:SS-} if and only if $f(x)=\alpha W_a(\lambda x)$ for some $\alpha\in\C$ and some $\lambda>0$.\\
{\upshape(ii)} The inequality \eqref{E:SS-} is valid also when $a=0$; however, equality now holds if and only if $f(x)=\alpha W_0(\lambda x+y)$ for some $\alpha\in\C$, some $y\in\R^d$,
and some $\lambda>0$.\\
{\upshape(iii)} If $a > 0$, then
\begin{equation}\label{E:SS+}
\|f\|_{L^{\frac{2d}{d-2}}_x(\R^d)} \leq \|W_0\|_{L^{\frac{2d}{d-2}}_x(\R^d)} \|W_0\|_{\dot H^{1}(\R^d)}^{-1} \|f\|_{\dot H^{1}_a(\R^d)}^{ }. 
\end{equation}
In this case, equality never holds (for $f\not\equiv 0$); however, the constant in \eqref{E:SS+} cannot be improved.
\end{proposition}

We will be proving part (iii) of this proposition via concentration compactness.  Before beginning the proof of Proposition~\ref{P:Wa}, we record the following elegant encapsulation of  the concentration compactness philosophy in the setting of classical Sobolev embedding due to G\'erard; see \cite{GerardESIAM}.   For our purposes, we need only consider radial (=spherically symmetric) functions; the paper \cite{GerardESIAM} treats the general case.  In Corollary~\ref{C:Bubble} below, we will adapt this to the $a\neq 0$ setting.

\begin{theorem}[Bubble decomposition for Sobolev embedding, the radial case]\label{T:Sob bubble}\leavevmode\quad\\
Fix $d\geq 3$ and let $\{f_n\}_{n\geq 1}$ be a sequence of radial functions, bounded in $\dot H^1_x(\R^d)$.  Then there
exist $J^*\in\{0,1,2,\ldots\}\cup\{\infty\}$, non-zero radial functions $\{\phi^j\}_{j=1}^{J^*}\subseteq \dot H^1_x$, and scales $\{\lambda_n^j\}_{j=1}^{J^*}\subseteq (0,\infty)$
so that along some subsequence in $n$ we may write
\begin{gather}\label{E:Sob bubble f=}
f_n(x) = \sum_{j=1}^J (\lambda_n^j)^{\frac{2-d}2} \phi^j\bigl(x/\lambda_n^j\bigr) + r_n^J(x) \quad\text{for all finite $\,0\leq J\leq J^*$}
\end{gather}
with the following five properties:
\begin{gather}\label{E:Sob bubble Lq}
\limsup_{J\to \infty} \limsup_{n\to\infty} \bigl\| r_n^J \bigr\|_{L^{\frac{2d}{d-2}}_x} =0 \\
\label{E:Sob bubble KE}
\sup_{J} \, \limsup_{n\to\infty}
    \biggl| \|f_n\|_{\dot H^1_x}^2 - \biggl( \| r_n^J \|_{\dot H_x^1}^2 + \sum_{j=1}^J \| \phi^j \|_{\dot H_x^1}^2 \biggr)\biggr| = 0\\
\label{E:Sob bubble PE}
\limsup_{J\to \infty} \limsup_{n\to\infty} \biggl| \bigl\|f_n\bigr\|_{L^{\frac{2d}{d-2}}_x}^{\frac{2d}{d-2}}
    - \sum_{j=1}^J \bigl\| \phi^j \bigr\|_{L^{\frac{2d}{d-2}}_x}^{\frac{2d}{d-2}}\biggr| =0 \\
\label{E:Sob bubble diverge}
\liminf_{n\to\infty} \biggl[
    \frac{\lambda_n^j}{\lambda_n^{j'}} + \frac{\lambda_n^{j'}}{\lambda_n^j} \biggr]= \infty \quad \text{for all $j\neq j'$}\\
\label{E:Sob bubble weak}
(\lambda_n^j)^{\frac{d-2}2} f_n \bigl( \lambda_n^j x\bigr) \rightharpoonup \phi^j(x)
    \quad\text{weakly in $\dot H_x^1$.}
\end{gather}
\end{theorem}

The concentration compactness principle recorded above can de adapted to the case $a\neq 0$, as follows:

\begin{corollary}\label{C:Bubble}
Fix $d\geq 3$ and $a>-(\frac{d-2}2)^2$.  Let $\{f_n\}_{n\geq 1}$ be a bounded radial sequence in $\dot H^1_a(\R^d)$.  Then there
exist $J^*\in\{0,1,2,\ldots\}\cup\{\infty\}$, non-zero radial functions $\{\phi^j\}_{j=1}^{J^*}\subseteq \dot H^1_x$, and scales $\{\lambda_n^j\}_{j=1}^{J^*}\subseteq (0,\infty)$
so that passing to a subsequence in $n$, equations \eqref{E:Sob bubble f=}, \eqref{E:Sob bubble Lq}, and \eqref{E:Sob bubble PE}--\eqref{E:Sob bubble weak} hold.  Moreover, in place of \eqref{E:Sob bubble KE}, we have
\begin{equation}\label{E:Bubble K}
\sup_{J} \, \limsup_{n\to\infty}  \biggl| \|f_n\|_{\dot H^1_a(\R^d)}^2 - \biggl( \| r_n^J \|_{\dot H_a^1(\R^d)}^2 + \sum_{j=1}^J \| \phi^j \|_{\dot H_a^1(\R^d)}^2 \biggr)\biggr| = 0.
\end{equation}
\end{corollary}

\begin{proof}
With the exception of \eqref{E:Bubble K}, all claims follow from the fact that $\dot H^1_a(\R^d)$ is isomorphic to $\dot H^1(\R^d)$, as seen in Theorem~\ref{pro:equivsobolev}.

In the setting of Theorem~\ref{T:Sob bubble}, the decoupling statement \eqref{E:Sob bubble KE} is proved by inductive application of the fact that in any Hilbert space $H$,
\begin{equation}\label{E:HilbDecoup}
g_n \rightharpoonup \phi \qtq{weakly in $H$ implies} \| g_n \|_H^2 -  \| g_n - \phi \|_H^2 - \| \phi \|_H^2 \longrightarrow 0.
\end{equation}
Specifically, one takes $H=\dot H^1(\R^d)$ and uses \eqref{E:Sob bubble weak}  and \eqref{E:Sob bubble diverge} as input.   However, weak convergence in $\dot H^1(\R^d)$ guarantees weak convergence in $\dot H^1_a(\R^d)$ and consequently, we may apply \eqref{E:HilbDecoup} with $H=\dot H^1_a(\R^d)$.   Applying this inductively yields \eqref{E:Bubble K}. 
\end{proof}

We are now ready to complete the proof of Proposition~\ref{P:Wa}.

\begin{proof}[Proof of Proposition~\ref{P:Wa}]
Part (ii) is precisely the sharp form of classical Sobolev embedding that has been known for some time \cite{Aubin,Bliss,Talenti}.  One proof of this, adapted to serve as an introduction to  the manner in which concentration compactness methods are used in dispersive PDE, can be found in \cite[\S4.2]{KVnote}.  We will be adapting that argument when we discuss part (i) below.

Part (iii) is a trivial consequence of part (ii): As $a>0$,
$$
 \|W_0\|_{L^{\frac{2d}{d-2}}_x(\R^d)}^{-1} \|W_0\|_{\dot H^{1}(\R^d)}^{\ }  \|f\|_{L^{\frac{2d}{d-2}}_x(\R^d)}  \leq \|f\|_{\dot H^{1}(\R^d)}^{ }
	< \|f\|_{\dot H^{1}_a(\R^d)}^{ }
$$
unless $f\equiv 0$.  The fact that the constant cannot be improved follows by considering $f_n(x) = W_0(x-x_n)$ for any sequence $x_n\to\infty$.

We turn now to the proof of part (i).  By the standard rearrangement inequalities (cf. \cite[Ch. 3]{LiebLoss}), the optimal constant can determined by the consideration of radial functions alone.  Moreover, as $\int |f(x)|^2 |x|^{-2}\,dx$ is \emph{strictly} monotone under rearrangement, any optimizer must in fact be radial.

Let $f_n$ be an optimizing sequence of radial functions for the problem
$$
\text{maximize} \quad J(f):=\|f\|_{L^{\!\frac{2d}{d-2}}_x}^\frac{2d}{d-2}\div \|f\|_{\dot H^1_a}^\frac{2d}{d-2}  \qtq{subject to the constraint} \|f\|_{\dot H^1_a}=1.
$$
Applying Corollary~\ref{C:Bubble} and passing to the requisite subsequence yields
\begin{equation}\label{E:SSEPf''}
\sup_f J(f) = \lim_{n\to\infty} J(f_n)
= \sum_{j=1}^{J^*} \bigl\| \phi^j \bigr\|_{L^{\!\frac{2d}{d-2}}_x}^{\frac{2d}{d-2}}
\leq \sup_f J(f) \sum_{j=1}^{J^*}\| \phi^j \|_{\dot H_a^1(\R^d)}^{\frac{2d}{d-2}},
\end{equation}
which implies
\begin{equation}\label{E:SSEPf}
1 \leq  \sum_{j=1}^{J^*} \| \phi^j \|_{\dot H_a^1(\R^d)}^{\frac{2d}{d-2}} .
\end{equation}
Corollary~\ref{C:Bubble} also guarantees that
\begin{equation}\label{E:SSEPf'}
\sum_{j=1}^{J^*} \| \phi^j \|_{\dot H_a^1(\R^d)}^2  \leq \limsup_{J\to J^*} \biggl( \| r_n^J \|_{\dot H_a^1(\R^d)}^2 + \sum_{j=1}^J \| \phi^j \|_{\dot H_a^1(\R^d)}^2 \biggr) =1.
\end{equation}
As $\frac{2d}{d-2} > 2$, we see that the consistency of \eqref{E:SSEPf} and \eqref{E:SSEPf'} requires that $J^*=1$ and that $\phi^1$ has unit $\dot H_a^1(\R^d)$ norm.  The fact that $f_n$ is an optimizing sequence shows that $\phi^1$ must be an optimizer for $J$ and hence for the embedding \eqref{E:SS-}.

Now that we know the existence of optimizers, we turn to their characterization.  Recall that we have already seen that optimizers must be radial. Let $0\not\equiv f\in \dot H^1_a$ denote such a radial optimizer.  Replacing $f$ by $\alpha f$ for some $\alpha>0$, if necessary, we may assume that
$$
  \|f\|_{\dot H^{1}_a(\R^d)}^2 = \int_{\R^d} |f(x)|^{\frac{2d}{d-2}}\,dx.
$$

By assumption,  $f$ maximizes $\int  |f(x)|^{2d/(d-2)}\,dx$ among all functions that share its $\dot H^{1}_a(\R^d)$ norm.  Thus, it obeys the Euler--Lagrange equation $\ld f = |f|^{4/(d-2)} f $;
here we exploited the normalization of $f$ to determine the Lagrange multiplier.  Corresponding to the scale-invariance of Sobolev embedding, Noether's theorem guarantees that this Euler--Lagrange equation admits a conservation law.  Proceeding in this way, we may rewrite the Euler--Lagrange equation (in the radial variable $r=|x|$) in the form
$$
-r^2\big\{\partial_r \big[r^{\frac{d-2}{2}} f(r) \big]\big\}^2+\bigl[ \bigl(\tfrac{d-2}{2}\bigr)^2 + a \bigr] \bigl[ r^{\frac{d-2}{2}} f(r)\bigr]^2 - \tfrac{d-2}{d} \bigl[ r^{\frac{d-2}{2}} f(r)\bigr]^{\frac{2d}{d-2}} = c.
$$
As $f\in \dot H^1_x$, there is a sequence $r_n\to\infty$ so that $|(\partial_r f)(r_n)|+|\frac{1}{r_n} f(r_n)| = o(r_n^{-d/2})$.  Thus, the constant $c$ is, in fact, zero.  The resulting first-order ODE is separable when written in the variables $w(r)=r^{\frac{d-2}{2}} f(r)$ and $\rho=\log(r)$.  Carrying out the requisite integrals, we then deduce that $f(x) = \lambda^{1/2} W_a(\lambda x)$ for some $\lambda>0$.
\end{proof}

If we discard the information regarding the existence of optimizers,  is not difficult to see that the three different cases in Proposition~\ref{P:Wa} can be combined into a single equivalent statement: 

\begin{corollary}[Sharp Sobolev embedding]\label{C:SSE'}
Fix $d\geq 3$ and $a>-(\frac{d-2}2)^2$.  Then
\begin{equation*}
\frac{2 E_a(f)}{\|W_{a\wedge 0}\|_{\dot H^1_{a\wedge 0}(\R^d)}^2}
	\geq \frac{\|f\|_{\dot H^1_{a}(\R^d)}^2}{\|W_{a\wedge 0}\|_{\dot H^1_{a\wedge 0}(\R^d)}^2}
			- \frac{d-2}{d} \Biggl[ \frac{\|f\|_{\dot H^1_{a}(\R^d)}^2}{\|W_{a\wedge 0}\|_{\dot H^1_{a\wedge 0}(\R^d)}^2}\Biggr]^{\frac{d}{d-2}} .
\end{equation*}
\end{corollary}

This formulation of sharp Sobolev embedding is well-suited to proving energy trapping and for demonstrating coercivity of the virial:

\begin{corollary}[Coercivity]\label{C: trap}
Fix $d\geq 3$ and $a>-(\frac{d-2}2)^2$.  Let $u:I\times\R^d \to \C$ be a solution to \eqref{E:fNLS} with initial data $u(t_0)=u_0\in \dot H^1_a(\R^d)$ for some $t_0\in I$. Assume
$E_a(u_0)\le (1-\delta_0) E_{a\wedge 0}(W_{a\wedge 0})$ for some $\delta_0>0$.  Then there exist positive constants $\delta_1$ and $c$ (depending on $d$, $a$, and $\delta_0$) such that

\noindent
{\upshape (a)} If $\|u_0\|_{\dot H^1_a(\R^d)}\le\|W_{a\wedge 0}\|_{\dot H^1_{a\wedge 0}(\R^d)}$, then for all $t\in I$
\begin{align*}
&(1)\ \|u(t)\|_{\dot H^1_a(\R^d)}\le (1-\delta_1)\|W_{a\wedge 0}\|_{\dot H^1_{a\wedge 0}(\R^d)}\\
&(2)\ \int_{\R^d}|\nabla u(t,x)|^2+ \tfrac{a}{|x|^2}|u(t,x)|^2-|u(t,x)|^{\frac{2d}{d-2}}\,dx \ge c \|u(t)\|_{\dot H^1_a}^2\\
&(3)\  c\|u(t)\|_{\dot H^1_a}^2\leq 2 E_a(u)\leq \|u(t)\|_{\dot H^1_a}^2.
\end{align*}

\noindent
{\upshape (b)} If  $\|u_0\|_{\dot H^1_a(\R^d)}\geq \|W_{a\wedge 0}\|_{\dot H^1_{a\wedge 0}(\R^d)}$, then for all $t\in I$
\begin{align*}
&(1)\ \|u(t)\|_{\dot H^1_a(\R^d)}\ge (1+\delta_1)\|W_{a\wedge 0}\|_{\dot H^1_{a\wedge 0}(\R^d)}\\
&(2)\ \int_{\R^d}|\nabla u(t,x)|^2+ \tfrac{a}{|x|^2}|u(t,x)|^2-|u(t,x)|^{\frac{2d}{d-2}}\,dx \le -c<0.
\end{align*}
\end{corollary}

\begin{proof}
Combining Corollary~\ref{C:SSE'} with the assumption $E_a(u_0)\le (1-\delta_0) E_{a\wedge 0}(W_{a\wedge 0})$ and \eqref{WaPoho}, we obtain
$$
\frac{2}{d} (1-\delta_0) \geq \frac{\|u(t)\|_{\dot H^1_{a}}^2}{\|W_{a\wedge 0}\|_{\dot H^1_{a\wedge 0}}^2}
			- \frac{d-2}{d} \Biggl[ \frac{\|u(t)\|_{\dot H^1_{a}}^2}{\|W_{a\wedge 0}\|_{\dot H^1_{a\wedge 0}}^2}\Biggr]^{\frac{d}{d-2}} .
$$
Claims (a)(1) and (b)(1) now follow by using a simple continuity argument together with the conservation of energy and the elementary inequality
$$
\tfrac{2}{d} (1-\delta_0) \geq y - \tfrac{d-2}{d} y^{\frac{d}{d-2}}  \quad\implies\quad |y-1| \geq \delta_1
$$
for some $\delta_1=\delta_1(d, \delta_0)$.

To verify items (a)(2) and (b)(2) we first write
\begin{align}\label{E:LHSofV}
\!\int_{\R^d}|\nabla u(t,x)|^2+ \tfrac{a}{|x|^2}|u(t,x)|^2-|u(t,x)|^{\frac{2d}{d-2}}\,dx = \tfrac{2d}{d-2} E_a(u) - \tfrac2{d-2} \|u\|_{\dot H^1_{a}}^2.
\end{align}
We will also need that by \eqref{WaPoho},
$$
\tfrac{2d}{d-2} E_{a\wedge0}(W_{a\wedge0}) = \tfrac2{d-2} \|W_{a\wedge0}\|_{\dot H^1_{{a\wedge0}}}^2 .
$$

In the setting of (b)(2), these two ingredients yield 
$$
\text{LHS\eqref{E:LHSofV}} \leq \tfrac{2d}{d-2} E_{a\wedge0}(W_{a\wedge0}) - \tfrac2{d-2} (1+\delta_1)^2 \|W_{a\wedge0}\|_{\dot H^1_{{a\wedge0}}}^2
	\leq - \tfrac{4\delta_1}{d-2} \|W_{a\wedge0}\|_{\dot H^1_{{a\wedge0}}}^2 < 0,
$$
which resolves this case.

In the setting of (a)(2), rearranging Corollary~\ref{C:SSE'} yields
$$
 \|u(t)\|_{\dot H^1_{a}(\R^d)}^2 - 2 E_a(u) \leq  \tfrac{d-2}{d}(1 - \delta_1)^{\frac4{d-2}}  \|u(t)\|_{\dot H^1_{a}(\R^d)}^2.
$$
In this way, we deduce that
$$
\text{LHS\eqref{E:LHSofV}} = \|u(t)\|_{\dot H^1_{a}}^2 - \tfrac{d}{d-2}\bigl[\|u(t)\|_{\dot H^1_{a}}^2 - 2 E_a(u) \bigr]
	\geq \bigl[ 1 - (1-\delta_1)^{\frac4{d-2}}\bigr] \|u(t)\|_{\dot H^1_{a}}^2.
$$

This leaves us to verify claim (a)(3).  This is not difficult.  As the nonlinearity is focusing, $2E_a(u)\leq \| u\|_{\dot H^1_a}^2$.  The other inequality follows from (a)(2) because $\text{LHS\eqref{E:LHSofV}} \leq 2E_a(u)$.  This finishes the proof of the lemma.
\end{proof}

Using the usual virial argument and Corollary~\ref{C: trap}(b), one can prove that finite-time blowup occurs for \eqref{E:fNLS} for $a>-(\frac{d-2}2)^2 + (\frac{d-2}{d+2})^2$; see Proposition~\ref{P:blowup} below.  The restriction on $a$ stems from the local well-posedness theory; it is needed to ensure that solutions can be constructed, at least locally in time.  The proof of Proposition~\ref{P:blowup} follows from a straight-forward adaptation of the arguments in, say, \cite[Section~9]{KV5}, relying on Corollary~\ref{C: trap}(b) for the requisite concavity.

\begin{proposition}[Blowup]\label{P:blowup}
Fix $d\geq 3$ and $a>-(\frac{d-2}2)^2 + (\frac{d-2}{d+2})^2$.   Let  $u_0\in \dot H_x^1(\R^d)$ be such that $E_a(u_0)< E_{a\wedge 0}(W_{a\wedge 0})$ and $\|u_0\|_{\dot H^1_a(\R^d)}\geq \|W_{a\wedge 0}\|_{\dot H^1_{a\wedge 0}(\R^d)}$.  Assume also that either $xu_0\in L^2_x(\R^d)$ or $u_0\in H_x^1(\R^d)$ is radial.  Then the corresponding solution $u$ to \eqref{E:fNLS} blows up in finite time.
\end{proposition}


\begin{thebibliography}{99}

\bibitem{AAR}
G. Andrews, R. Askey, and R. Roy, \emph{Special functions.} Encyclopedia of Mathematics and its Applications, \textbf{71}. Cambridge University Press, Cambridge, 1999. 

\bibitem{Aubin}
T. Aubin, \emph{Probl\'emes isop\'erim\'etriques et espaces de Sobolev.} J. Diff. Geom. \textbf{11} (1976), 573--598.

\bibitem{BahouriGerard} H. Bahouri and P. G\'erard, \emph{High frequency approximation of solutions to critical nonlinear wave equations.}
Amer. J. Math. \textbf{121} (1999), no. 1, 131--175.

\bibitem{Bliss}
G. A. Bliss, \emph{An integral inequality.} J. London Math. Soc. \textbf{5} (1930), 40--46.


\bibitem{Bo99a}
J. Bourgain, \emph{Global well-posedness of defocusing 3D critical NLS in the radial case}. J. Amer. Math. Soc. \textbf{12} (1999), 145--171.

\bibitem{BrezisLieb}
H.~Br\'ezis and E.~Lieb, \emph{A relation between pointwise convergence of functions and convergence of functionals.} Proc.
Amer. Math. Soc. \textbf{88} (1983), 486--490.

\bibitem{BPST1}
N. Burq, F. Planchon, J. Stalker, and A. S. Tahvildar-Zadeh, \emph{Strichartz estimates for the wave and Schr\"odinger equations with the inverse-square potential.} J.
Funct. Anal. \textbf{203} (2003), 519--549.

\bibitem{BPST2}
N. Burq, F. Planchon, J. Stalker, and A. S. Tahvildar-Zadeh, \emph{Strichartz estimates for the wave and Schr\"odinger equations with potentials of critical decay.}
Indiana Univ. Math. J. \textbf{53} (2004), 1665--1680.



\bibitem{CK}
M. Christ and A. Kiselev, \emph{Maximal functions associated to filtrations.} J. Funct. Anal. \textbf{179} (2001), 409--425.

\bibitem{ChW:fractional chain rule}
M.~Christ and M.~Weinstein, \emph{Dispersion of small amplitude solutions of the generalized Korteweg--de Vries equation.}
J. Funct. Anal. \textbf{100} (1991), 87--109.

\bibitem{CKSTT07}
J. Colliander, M. Keel, G. Staffilani, H. Takaoka, and T. Tao, \emph{Global well-posedness and scattering for the energy-critical nonlinear Schr\"odinger equation in $\R^3$.}
Annals of Math. \textbf{167} (2008), 767--865.

\bibitem{Dod} B. Dodson, \emph{Global well-posedness and scattering for the focusing, energy-critical nonlinear Schr\"odinger problem in dimension $d=4$ for initial data below a ground state threshold.} Preprint \texttt{arXiv:1409.1950}.



\bibitem{Fanelli}
L. Fanelli, V. Felli, M. A. Fontelos, and A. Primo, \emph{Time decay of scaling critical electromagnetic Schr\"odinger flows.} 
Comm. Math. Phys. \textbf{324} (2013), no. 3, 1033--1067.

\bibitem{GerardESIAM}
P.~G\'erard, \emph{Description du d\'efaut de compacit\'e de l'injection de Sobolev.} ESAIM Control Optim. Calc. Var. \textbf{3} (1998), 213--233.




\bibitem{IonPaus1}
A. D. Ionescu and B. Pausader, \emph{Global well-posedness of the energy-critical defocusing NLS on $\R\times\mathbb{T}^3$.}
Comm. Math. Phys. \textbf{312} (2012), no. 3, 781--831.

\bibitem{IonPaus}
A. D. Ionescu and B. Pausader, \emph{The energy-critical defocusing NLS on $\mathbb{T}^3$.}
Duke Math. J. \textbf{161} (2012), no. 8, 1581--1612.

\bibitem{IPS:H3}
A. D. Ionescu, B. Pausader, and G. Staffilani, \emph{On the global well-posedness of energy-critical Schr\"odinger equations in curved spaces.}
Anal. PDE \textbf{5} (2012), no. 4, 705–--746.

\bibitem{Jao1} C. Jao, \emph{The energy-critical quantum harmonic oscillator.}  Preprint \texttt{arXiv:1406.2289}.

\bibitem{Jao2} C. Jao, \emph{Energy-critical NLS with potentials of quadratic growth.} Preprint \texttt{arXiv:1411.4950}

\bibitem{KSWW}
H. Kalf, U. W. Schmincke, J. Walter, and R. W\"ust, \emph{ On the spectral theory of Schr\"odinger and Dirac operators with strongly singular potentials. In Spectral theory and differential equations.} 182--226. Lect. Notes in Math. \textbf{448} (1975) Springer, Berlin.

\bibitem{KeelTao}
M. Keel and T. Tao, \emph{Endpoint Strichartz estimates.} Amer. J. Math. \textbf{120} (1998), 955--980.

\bibitem{KM}
C. Kenig and F. Merle, \emph{Global well-posedness, scattering, and blow-up for the energy-critical focusing nonlinear Schr\"{o}dinger equation in the radial case.} Invent. Math. \textbf{166} (2006), 645--675.

\bibitem{Keraani} S. Keraani, \emph{On the blow up phenomenon of the critical nonlinear Schr\"odinger equation.} 
J. Funct. Anal. \textbf{235} (2006), no. 1, 171--192.

\bibitem{KKSV:gKdV}
R. Killip, S. Kwon, S. Shao, and M. Visan, \emph{On the mass-critical generalized KdV equation.}	
DCDS-A \textbf{32} (2012), 191--221.

\bibitem{KMVZZ} R. Killip, C. Miao, M. Visan, J. Zhang, and J. Zheng,
\emph{Multipliers and Riesz transforms for the Schr\"odinger operator with inverse-square potential.} Preprint \texttt{arXiv:1503.02716}.

\bibitem{KOPV} R. Killip, T. Oh, O. Pocovnicu, and M. Visan, \emph{Solitons and scattering for the cubic-quintic nonlinear Schr\"odinger equation on $\R^3$.}  Preprint \texttt{arXiv:1409.6734}.	

\bibitem{KSV:2DKG}
R. Killip, B. Stovall, and M. Visan, \emph{Scattering for the cubic Klein--Gordon equation in two space dimensions.}
Trans. Amer. Math. Soc. \textbf{364} (2012), 1571--1631.

\bibitem{KV5} R. Killip and M. Visan, \emph{The focusing energy-critical nonlinear Schrödinger equation in dimensions five and higher.}
Amer. J. Math. \textbf{132} (2010), no. 2, 361--424.

\bibitem{KVnote}
R. Killip and M. Visan, \emph{Nonlinear Schr\"odinger equations at critical regularity.} In ``Evolution equations'', 325--437, 
Clay Math. Proc., \textbf{17}. Amer. Math. Soc., Providence, RI, 2013.

\bibitem{KV:quintic}
R. Killip and M. Visan, \emph{Global well-posedness and scattering for the defocusing quintic NLS in three dimensions.} 
Anal. PDE \textbf{5} (2012), no. 4, 855--885. 

\bibitem{KVZ12} R. Killip, M. Visan, and X. Zhang, \emph{Quintic NLS in the exterior of a strictly convex obstacle.}
Preprint \texttt{arXiv:1208.4904} to appear in Amer. J. Math.

\bibitem{Oberbook}
H. Koch, D. Tataru, and M. Visan, \emph{Dispersive Equations and Nonlinear Waves}, Oberwolfach Seminars, \textbf{45}. Birkhauser/Springer Basel AG, Basel, 2014.

\bibitem{LiebLoss} E. Lieb and M. Loss, \emph{Analysis.}  Second edition. Graduate Studies in Mathematics, \textbf{14}. American Mathematical Society, Providence, RI, 2001.

\bibitem{LS} V. Liskevich and Z. Sobol, \emph{Estimates of integral kernels for semigroups associated with second order elliptic
operators with singular coefficients.} Potential Anal. \textbf{18}
(2003), 359--390.

\bibitem{MS} P. D. Milman and Yu. A. Semenov, \emph{Global heat kernel bounds via desingularizing weights.} J. Funct. Anal. \textbf{212} (2004), 373--398.


\bibitem{PausTzW} B. Pausader, N. Tzvetkov, and X. Wang, \emph{Global regularity for the energy-critical NLS on $\mathbb{S}^3$.}
Ann. Inst. H. Poincar\'e Anal. Non Lin\'eaire \textbf{31} (2014), no. 2, 315--338.


\bibitem{PSS1} F. Planchon, J. Stalker, and A. S. Tahvildar-Zadeh, \emph{Dispersive estimates for wave equation with the inverse-square potential.}
Discrete Contin. Dynam. Systems, \textbf{9} (2003), 1387--1400.

\bibitem{RS2}  M. Reed and B. Simon, \emph{Methods of modern mathematical physics. II. Fourier analysis, self-adjointness.} Academic Press, New York-London, 1975.




\bibitem{RV}
E.~Ryckman and M.~Visan, \emph{Global well-posedness and scattering
for the defocusing energy-critical nonlinear Schr\"odinger equation
in $\R^{1+4}$.} Amer. J. Math. \textbf{129} (2007), 1--60.




\bibitem{Talenti} G. Talenti, \emph{Best constant in Sobolev inequality.} Ann. Mat. Pura. Appl. \textbf{110} (1976), 353--372.

\bibitem{TaoRadial} T. Tao, \emph{Global well-posedness and scattering for higher-dimensional energy-critical non-linear
Schr\"odinger equation for radial data.} New York J. of Math. \textbf{11} (2005), 57--80.


\bibitem{TaoVisan}
T. Tao and M.~Visan, \emph{Stability of energy-critical nonlinear Schr\"odinger equations in high dimensions.} Electron. J. Diff. Eqns. \textbf{118} (2005), 1--28.



\bibitem{VZ}
J. L. Vazquez and E. Zuazua, \emph{The Hardy inequality and the asymptotic behaviour of the heat equation with an inverse-square potential.}
J. Funct. Anal. \textbf{173} (2000), 103--153.

\bibitem{Visanphd}
M. Visan, \emph{The defocusing energy-critical nonlinear Schr\"odinger equation in dimensions five and higher.} Ph. D Thesis, UCLA, 2006.

\bibitem{Visan2007}
M. Visan, \emph{The defocusing energy-critical nonlinear Schr\"odinger equation in higher dimensions.} Duke Math. J. \textbf{138} (2007)
281--374.


\bibitem{ZZ} 
J. Zhang and J. Zheng, \emph{Scattering theory for nonlinear Schr\"odinger with inverse-square potential.}  J. Funct. Anal. \textbf{267} (2014), 2907--2932.

\end{thebibliography}
\end{document}